\documentclass[a4paper,11pt]{article}

\usepackage{mathrsfs}
\usepackage{graphicx}
\usepackage{amsmath}
\usepackage{amsfonts}
\usepackage{amssymb}
\usepackage{amsthm}
\usepackage{textcomp}
\usepackage{color}
\usepackage{geometry}
\usepackage[normalem]{ulem}
\usepackage{placeins}
\usepackage{bm}
\usepackage{wasysym}
\usepackage{titlesec}

\newtheorem{theorem}{Theorem}[section]
\newtheorem{lemma}{Lemma}[section]
\newtheorem{remark}{Remark}[section]

\renewcommand{\theequation}{\arabic{section}.\arabic{equation}}
\renewcommand{\thetheorem}{\arabic{section}.\arabic{theorem}}
\renewcommand{\thelemma}{\arabic{section}.\arabic{lemma}}
\renewcommand{\theproposition}{\arabic{section}.\arabic{proposition}}

\renewcommand{\thealgorithm}{\arabic{section}.\arabic{algorithm}}

\newcommand\dd{~{\rm d}}
\newcommand{\burg}{{\sf b}}

\newcommand\R{\mathbb{R}}
\newcommand\Rg{\mathcal{R}}
\newcommand\N{\mathbb{N}}
\newcommand\Z{\mathbb{Z}}

\newcommand\E{\mathcal{E}}

\newcommand\F{\mathcal{F}}
\newcommand\bg{{\bm g}}
\newcommand\Wdot{\dot{\mathscr{W}}^{1,2}}

\newcommand\del{\delta}
\newcommand\ddel{\delta^2}

\newcommand{\<}{\langle}
\renewcommand{\>}{\rangle}

\newcommand{\ulin}{u^{\rm lin}}

\def\Xint#1{\mathchoice
	{\XXint\displaystyle\textstyle{#1}}%
	{\XXint\textstyle\scriptstyle{#1}}%
	{\XXint\scriptstyle\scriptscriptstyle{#1}}%
	{\XXint\scriptscriptstyle\scriptscriptstyle{#1}}%
	\!\int}
\def\XXint#1#2#3{{\setbox0=\hbox{$#1{#2#3}{\int}$ }
		\vcenter{\hbox{$#2#3$ }}\kern-.6\wd0}}
\def\mint{\Xint-}

\titleformat{\subsubsection}[runin]{\bf }{
  \arabic{section}.\arabic{subsection}.\arabic{subsubsection}}{
  0.5em}{}[.]

\title{QM/MM Methods for Crystalline Defects \\
  Part 2 : Consistent Energy and Force-Mixing}

\author{Huajie Chen and Christoph Ortner\footnote{{\tt
      huajie.chen@warwick.ac.uk} and {\tt c.ortner@warwick.ac.uk}.  Mathematics
    Institute, University of Warwick, Coventry CV47AL, UK.  This work was
    supported by ERC Starting Grant 335120. CO's work was also supported by the
    Leverhulme Trust through a Philip Leverhulme Prize and by EPSRC Standard
    Grant EP/J021377/1.  }}

\date{}

\begin{document}
\maketitle

\begin{abstract}
  QM/MM hybrid methods employ accurate quantum (QM) models only in regions of
  interest (defects) and switch to computationally cheaper interatomic potential
  (MM) models to describe the crystalline bulk.
  
  We develop two QM/MM hybrid methods for crystalline defect simulations, an
  energy-based and a force-based formulation, employing a tight binding QM
  model. Both methods build on two principles: (i) locality of the QM model;
    and (ii) constructing the MM model as an explicit and controllable
    approximation of the QM model. This approach enables us to establish
  explicit convergence rates in terms of the size of QM region.
\end{abstract}

\section{Introduction}\label{sec-introduction}
Algorithms for concurrently coupling quantum mechanics and classical molecular
mechanics (QM/MM) are widely used to perform simulations of large systems in
materials science and biochemistry \cite{bernstein09, csanyi04, gao02,
  kermode08, ogata01, warshel76, zhang12}.  A QM model is necessary for accurate
treatments of bond breaking/formation, charge transfer, electron excitation and
so on. However, the applications of QM is limited to systems with hundreds of
atoms due to the significant computational cost. By contrast, MM methods based
on empirical inter-atomic potentials are able to treat millions of atoms or more
but reduced accuracy (more precisely, they are not transferable).  QM/MM
coupling methods promise (near-)QM accuracy at (near-)MM computational cost for
large-scale atomistic simulations in materials science.
% They combine
% the accuracy of QM description with the low computational cost of MM models,
% which have become popular in the past decades

In QM/MM simulations the computational domain is partitioned into two regions.
The region of primary interest, described by a QM model, is embedded in an
environment (e.g., bulk crystal) which is described by an MM model.  The
coupling between these two regions is the key challenge in the construction of
accurate and efficient QM/MM methods.

% Generally speaking, QM/MM coupling schemes can be classified according to
% whether they combine the QM and MM regions on the level of energies or forces
% \cite{bernstein09}: the energy-based methods build a combined total energy
% function of the systems; while the force-based methods combine QM forces for
% atoms in the QM region with MM forces in the MM region, possibly with an
% interpolation between the two in a transition region.  We will discuss both
% types of the methods in this paper.

A natural question is the accuracy of QM/MM models as a function of QM region
size. % With the hybrid schemes, the system should behave as if it was
% modelled fully quantum mechanically.  
The number of atoms in the QM region is a discretisation parameter and the
observables of interest should converge to the desired accuracy with respect to
this parameter.  Despite the growing number of QM/MM methods and their
applications, few of the publications have included quantitative tests of the
accuracy of the method and its convergence with respect to possible parameters.
To our best knowledge, there is no theoretical analysis for QM/MM methods in the
literature.

The purpose of this paper is to initiate a numerical analysis of QM/MM
methods. We develop two new QM/MM methods for crystalline defect simulations for
which we can prove rigorous a priori error estimates. We use the tight binding
(TB) model (a minimalist QM method) as the QM model and, for the MM region, {\em
  construct} an interatomic potential (or, forces) through an explicit
approximation of the TB model, which is reminiscent of the force matching
technique \cite{ercolessi94}. This approach enables us to establish explicit
convergence rates in terms of the size of the QM region.

Our analysis is based on two key preliminaries: the ``strong locality'' of the
(finite temperature) tight binding model \cite{chen15a} and the decay estimates
of equilibria in lattices with defects \cite{chenpre_vardef,ehrlacher13}.

\subsubsection*{\bf Outline}
In Section \ref{sec-review}, we review the existing QM/MM methodology for
material systems.  In Section \ref{sec-pre}, we review the tight binding model
for crystalline defects which we use as the QM model.  In Section
\ref{sec-e-mix-idea} and \ref{sec-f-mix-idea} we construct QM/MM coupling
schemes with rigorous error estimates based, respectively, on energy-mixing and
force-mixing principles. Finally, we summarise our findings and make some
concluding remarks concerning practical aspects which we will pursue in
forthcoming work.

\subsubsection*{Notation}
We will use the symbol $\langle\cdot,\cdot\rangle$ to denote an abstract duality
pairing between a Banach space and its dual. The symbol $|\cdot|$ normally
denotes the Euclidean or Frobenius norm, while $\|\cdot\|$ denotes an operator
norm.
For the sake of brevity of notation, we will denote $A\backslash\{a\}$ by
$A\backslash a$, and $\{b-a~\vert ~b\in A\}$ by $A-a$.

For a differentiable function $f$, $\nabla f$ denotes the Jacobi matrix.
% and $\nabla_{r}f = \nabla f\cdot r$ the directional derivative.
For $E \in C^2(X)$, the first and second variations are denoted by
$\< \delta E(u), v \>$ and $\< \delta^2 E(u) w, v \>$ for $u, v, w \in X$.
% We will avoid use of higher variations in this explicit way.
For higher variations, we will use the notation
$\delta^k E(u_0)[u_1,\cdots,u_k]$, and $\delta^k E(u_0)[u^{\otimes k}]$ for
abbreviation when $u_1=\cdots=u_k = u$.

For $j\in\N$, $\bg\in(\R^d)^{A}$, and $V\in C^j\big((\R^d)^{A}\big)$, we define
the notation
\begin{eqnarray*}
	V_{,{\bm \rho}}\big({\bm g}\big) := 
	\frac{\partial^j V\big({\bm g}\big)}
	{\partial {\bm g}_{\rho_1},\cdots,\partial{\bm g}_{\rho_j}}
	\quad{\rm for}~~
	{\bm \rho}=(\rho_1,\cdots,\rho_j)\in A^j.
\end{eqnarray*}	

The symbol $C$ denotes generic positive constant that may change from one line
of an estimate to the next. When estimating rates of decay or convergence, $C$
will always remain independent of the system size, of lattice position or of
test functions. The dependencies of $C$ will normally be clear from the context
or stated explicitly.

 % ACKNOWLEDGEMENTS
\subsubsection*{Acknowledgements}
We thank Noam Bernstein, Gabor Cs\'{a}nyi and James Kermode for their helpful
discussions. The work presented here is related to ongoing joint work.

\section{QM/MM coupling methods}
\label{sec-review}
\setcounter{equation}{0}
The many different variants to couple QM and MM models can be broadly divided
into two categories: energy-mixing and force-mixing. Energy-mixing methods
define an energy functional that involves mixture of QM and MM energies, and the
solution is obtained by minimizing the energy functional.  By contrast,
force-mixing methods define a system of force balance equations, where the forces
involve contributions from QM and MM models and are non-conservative (i.e., they
are not compatible with any energy functional). In the present section we review
both classes of QM/MM schemes for materials systems.

\subsection{Energy-mixing}\label{sec-e-mix}
\label{sec-e-mix}
The system under consideration is partitioned into QM and MM regions (see Figure
\ref{fig-qmmm-dec} for two examples).  Let $y^{\rm QM}$ and $y^{\rm MM}$ denote
the respective atomic configurations in these two regions.  Depending on the
construction of the hybrid total energy $E$, energy-based methods can mainly be
divided into two categories:

\begin{figure}[ht]
	\centering
	\includegraphics[height=6.0cm]{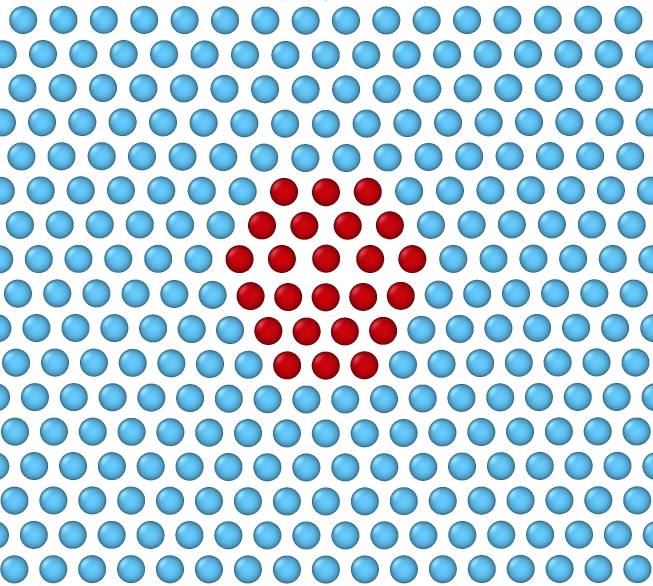} \qquad
        \includegraphics[height=6.0cm]{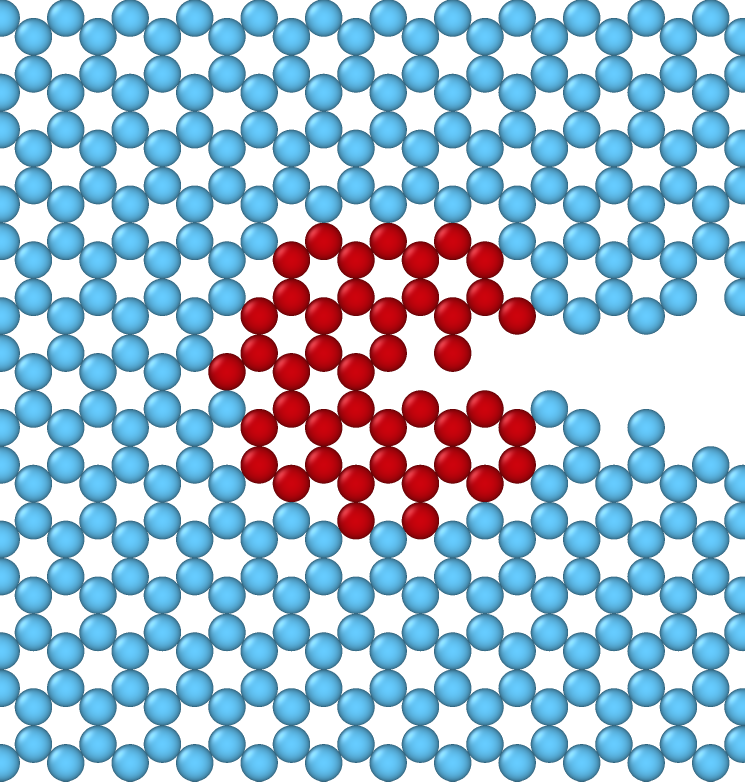}
	\caption{Partition of QM and MM regions for an edge dislocation in a 2D
          triangular lattice and crack in the 2D hexagonal lattice
          (cartoons). The dislocation core and a small neighbourhood belong to the 
          QM region (red / dark), while the bulk crystal behaves purely
          elastically and can therefore be well described by an empirical
          interatomic potential (blue / light).
          % silicon fracture simulation, where the crack tip and its surrounding
          % belong to QM region.  Reprinted with permission from
          % \cite{bernstein09,broughton99}
        }
	\label{fig-qmmm-dec}
\end{figure}

(1) In the {\it subtractive approach}, e.g. the {\it ONIOM} method
\cite{khare07,ogata01,svensson96} and its derivatives,
\begin{eqnarray}\label{energy-subtrative}
E(y^{{\rm QM}\cup{\rm MM}}) = E^{\rm MM}(y^{{\rm QM}\cup{\rm MM}}) 
+ E^{\rm QM}(y^{\rm QM}) - E^{\rm MM}(y^{\rm QM});
\end{eqnarray}
that is, an MM energy for the entire system is corrected by the difference
between the QM and MM energies of the QM region. 

(2) In the {\it additive approach}, e.g. {\it ChemShell} \cite{metz14}, {\it
  DL-FIND} \cite{kastner09}, {\it MAAD} \cite{broughton99} and {\it QUASI}
\cite{sherwood03}) the QM energy of the QM region and MM energy of the MM region
are connected via an interaction energy which may depend on an interface that
involves parts of both regions,
\begin{eqnarray}\label{energy-additive}
  E(y^{{\rm QM}\cup{\rm MM}}) = E^{\rm QM}(y^{\rm QM}) 
  + E^{\rm MM}(y^{\rm MM}) + E^{\rm interaction}(y^{\rm QM},y^{\rm MM}).
\end{eqnarray}

The advantage of energy-based methods is that they are naturally energy
conserving. Unfortunately, the spurious interface effects acting between the QM
and MM regions can be significant.
% A main challenge of the energy-based 
% methods is to minimize the boundary effects. 
To alleviate such effects, the QM and MM regions are either ``passivated''
\cite{sherwood03,sokol04} or ``buffered'' \cite{ogata05,ogata01}. In the first
approach, the energies $E^{\rm QM}(y^{\rm QM})$ in \eqref{energy-subtrative} and
\eqref{energy-additive} are the energies of the passivated cluster of the QM
region, in which a number of additional atoms that have no counterparts in the
real system (for example, hydrogen atoms) are added to the QM region to
terminate the broken bonds.  The second approach handles the boundary by
defining buffer layers surrounding the QM and MM regions, so that each atom can
see a full complement of surrounding atoms. 

The second approach seems to be preferred in solid state systems since the
elimination of the boundary effects for passivated atoms will not be perfect and
may indeed be severe \cite{csanyi05}.  The simplest example is for a perfect
bulk system, where the true force on all atoms are zero. However, the passivated
cluster force computed with QM and MM will in general be non-zero on the
passivation atoms and nearby atoms \cite{bernstein09}. This is reminiscent of
the ghost forces which are a well-understood concept in atomistic/continuum
multi-scale methods \cite{luskin13}.

\subsubsection{An idealized hybrid model}
\label{sec-e-idealized}
Beside the widely used {\it subtractive} and {\it additive} approaches, there is
a third type of energy-based formulations: the {\it local energy approach},
which mixes local energies computed by QM and MM methods in their respective
regions,
\begin{eqnarray}\label{energy-local}
E(y^{{\rm QM}\cup{\rm MM}}) = 
\sum_{\ell\in{\rm QM}}E_{\ell}^{\rm QM}(y^{{\rm QM}\cup{\rm MM}}) 
+ \sum_{\ell\in{\rm MM}}E_{\ell}^{\rm MM}(y^{{\rm QM}\cup{\rm MM}})
\end{eqnarray}
with $E_{\ell}$ denoting the local energy associated to the $\ell$-th atomic
site.  Even though the expression \eqref{energy-local} seems intuitive, this
variant is not commonly used in QM/MM coupling schemes. Indeed, we are only
  aware of a brief reference to this approach in \cite{bernstein09}. The reason
  is that it was unclear how to decompose $E$ into local contributions
  $E_\ell^{\rm QM}$ that match with a classical interatomic potential site
  energy $E_\ell^{\rm MM}$. 

In the previous work of this series \cite{chen15a}, we
studied the tight binding site energy introduced in \cite{ercolessi05,finnis03}
and justified its ``strong locality" rigorously. This is important and useful in
QM/MM coupling scheme based on \eqref{energy-local} since: (1) when using
classical potentials the total energy is almost always written as a sum over
atoms $E=\sum_{\ell}E_{\ell}$, therefore, we are able to establish the bridge
between the electronic and classical worlds; (2) rather than using ``black-box"
MM potentials, we can construct MM site energies based on the approximations of
QM site energies for good coupling of the different models. 

It is pointed out in \cite{bernstein09} that matching the force-constant/dynamic
matrix (i.e. the first derivatives of the force or the second order derivatives
of the energy with respect to atomic positions) would guarantee a perfect match
between the MM and QM forces for arbitrary infinitesimal displacements from
equilibrium. In case of an energy-based method, only such a strict matching
criterion can guarantee that spurious forces are eliminated for near equilibrium
configurations. The errors resulting from mismatching the force-constant/dynamic
matrix are analogous to so-called ``ghost forces" in atomistic/continuum hybrid
schemes.

Based on these observations, we construct an idealized MM site energy by taking
the second order expansion of $E_{\ell}^{\rm QM}$:
\begin{eqnarray}\label{site-MM}
E_{\ell}^{\rm lin}(y) := E_{\ell}^{\rm QM}(y_0) 
+ \big\<\delta E_{\ell}^{\rm QM}(y_0),y-y_0\big\> 
+ \frac{1}{2}\big\<\delta^2 E_{\ell}^{\rm QM}(y_0)(y-y_0),y-y_0\big\>,
\end{eqnarray}
where $y=y^{{\rm QM}\cup{\rm MM}}$ and $y_0$ is a predicted (near-)equilibrium
configuration, typically the far-field crystalline environment or an explicit
linearlised elasticity solution. The QM/MM total energy \eqref{energy-local} is
then given by
\begin{eqnarray}\label{energy-qmmm-lin}
E(y) = \sum_{\ell\in{\rm QM}}E_{\ell}^{\rm QM}(y) 
+ \sum_{\ell\in{\rm MM}}E_{\ell}^{\rm lin}(y).
\end{eqnarray} 
\eqref{energy-qmmm-lin} gives rise to a simple QM/MM coupling scheme, in which
the MM potential is constructed such that it matches the QM model.  Matching of
the MM and QM models for higher order information can also become important,
e.g., for slowly decaying elastic fields (dislocations, cracks) or due to
increased temperature which may cause fluctuations to displacements beyond the
quadratic regime \cite{bernstein09}. We will therefore discuss arbitrary order
expansions in this paper.

\subsection{Force-mixing}
\label{sec-f-mix}
With the partition of QM and MM regions, the force-based methods combine QM
forces for atoms in the QM region with MM forces in the MM region. 
The simplest variant, {\it brutal force mixing} \cite{csanyi05}, is defined by
\begin{eqnarray}
  F_{\ell}(y^{{\rm QM}\cup{\rm MM}}) = \left\{\begin{array}{ll}
F^{\rm QM}_{\ell}(y^{{\rm QM}\cup{\rm MM}}) & {\rm if}~\ell\in{\rm QM} \\[1ex]
F^{\rm MM}_{\ell}(y^{{\rm QM}\cup{\rm MM}}) & {\rm if}~\ell\in{\rm MM}
\end{array}\right. .
\end{eqnarray}
Typically, the QM and MM forces are computed by carving a cluster buffered by a
layer of atoms defined by a distance cutoff, and the forces on all the buffer
atoms are discarded \cite{csanyi05}. In other variants, a transition region is
introduced where forces are smoothly blended \cite{csanyi05}. Examples of
force-mixing QM/MM methods are {\it DCET} \cite{abraham00,bernstein03} and {\it
  LOTF} \cite{csanyi04,vita98}.

The main advantage of force-mixing is that there are no spurious interface
forces as in energy-mixing
schemes. % (double-counted or missing contributions to
% the total energy and incorrect energies contributed by passivation or buffer
% atoms). 
However, this comes at the cost of a non-conservative force field. Moreover, if
the QM and MM forces are directly used (without modification) for molecular
dynamics simulations, then the dynamics will not conserve momentum
\cite{bernstein09}.

\subsubsection{An idealized model}
\label{sec-f-idealized}
Similar to the discussions for energy-mixing methods, we will construct MM
forces by an expansion of QM forces so that the force-constant/dynamic matrix
can be matched, e.g., with a first order expansion,
\begin{eqnarray}
  F_{\ell}^{\rm MM}(y) := F_{\ell}^{\rm lin}(y) :=
  F_{\ell}^{\rm QM}(y_0) + \big\<\delta F_{\ell}^{\rm QM}(y_0),y-y_0\big\>,
\end{eqnarray}
where $y=y^{{\rm QM}\cup{\rm MM}}$ and $y_0$ is a suitable predictor of the
equilibrium configuration. For the same reasons as in the energy-mixing
approach, we will also consider higher-order expansions of the forces.

\begin{remark}
  Our construction of the MM site energies and of the MM forces is reminiscent
  of the classical idea of {\em force matching}. This is usually applied to the
  construction of interatomic potentials \cite{ercolessi94} and has more
  recently been applied in a coupling context, e.g., in
  \cite{csanyi04,tejada15}. A key difference in our present work, in the
  energy-based variant, is that we match the site energies rather than that
  total energies (and forces).
\end{remark}

\section{Tight binding model for crystalline defects}
\label{sec-pre}
\setcounter{equation}{0}
A finite or countable index-set $\Lambda \subset \R^m$ is called a {\em
  reference configuration}. A deformed configuration is described by a map
$y : \Lambda\rightarrow\R^d$ with $m, d \in \{2,3\}$ denoting the space
dimensions. (Allowing $m \neq d$ allows us to define 2D models of straight
dislocations.)

We say that the map $y$ is a {\it proper configuration} if the atoms do not
accumulate:
\begin{flushleft} 
	{\bf L}. \quad $\exists~\mathfrak{m}>0$ such that
	~~$|y(\ell)-y(k)|\geq\mathfrak{m} |\ell-k|
	~~\forall~\ell,k\in\Lambda$.
\end{flushleft}
Throughout, we let
$\mathcal{V}_{\mathfrak{m}}\subset\big(\mathbb{R}^d\big)^{\Lambda}$ denote the
subset of all $y \in (\R^d)^{\Lambda}$ satisfying {\bf L}.  If we need to
emphasize the domain $\Lambda$, then we will write
$\mathcal{V}_{\frak m}(\Lambda)$.

\subsection{The tight binding model and its site energy}
\label{sec-tb}
\def\Rc{R_{\rm c}}
The tight binding model is a minimalist QM type model, which enables the
investigation and prediction of properties of molecules and materials.  For
simplicity of presentation, we consider a `two-centre' tight binding model
\cite{goringe97,Papaconstantopoulos15} with a single orbital per atom and the
identity overlap matrix.  All results can be extended directly to general
non-self-consistent tight binding models, as described in~\cite{chen15a}.

For a finite system with reference configuration $\Omega$, $\#\Omega = N$,
the `two-centre' tight binding model is formulated in terms of a discrete
Hamiltonian, with the matrix elements
\begin{eqnarray}\label{tb-H-elements}
\Big(\mathcal{H}(y)\Big)_{\ell k}
=\left\{ \begin{array}{ll} % \displaystyle
           h_{\rm ons}\left(\sum_{j\neq \ell} 
           \varrho\big(|y({\ell})-y(j)|\big)\right) 
	%\big(\{|y(\ell)-y(n)|\}_{|y(\ell)-y(n)|\leq \Rc}\big)
	& {\rm if}~\ell=k \\[1ex]
	h_{\rm hop}\big(|y(\ell)-y(k)|\big) & {\rm if}~\ell\neq k,
	%~{\rm and}~|y(\ell)-y(k)|\leq \Rc; 
	% \\[1ex ]
	% 0 & {\rm otherwise},
\end{array} \right. 
\end{eqnarray}
where $\Rc$ is a cut-off radius,
$h_{\rm ons} \in C^{\mathfrak{n}}([0, \infty))$ is the on-site term, with
$\varrho \in C^{\mathfrak{n}}([0, \infty))$, $\varrho = 0$ in $[\Rc,\infty)$,
and $h_{\rm hop} \in C^{\mathfrak{n}}([0, \infty))$ is the hopping term with
$h_{\rm hop}(r)=0~\forall r\in[\Rc,\infty)$. 

Our results can be generalised to the more general TB model presented in
  \cite{chen15a}, but for the sake of simplicity of notation, we restrict
  ourselves to \eqref{tb-H-elements}, which still includes most
  non-self-consistent TB models in the literature.

%
% $h_{\rm ons}$ depends on the neighbouring atomic configuration and is invariant
% with respect to the permutation of $n$, and the hopping term $h_{\rm hop}$
% depends only on the displacement between the two atoms.
% Here, we have ignored the indices for multiple atomic orbitals and overlap
% matrix for simplicity of notations (see \cite{chen15a} for more detailed
% discussions).
Note that $h_{\rm ons}$ and $h_{\rm hop}$ are independent of $\ell$ and $k$,
which indicates that all atoms of the system belong to the same species.
We observe that the formulation \eqref{tb-H-elements} satisfies all the
assumptions on Hamiltonian matrix elements in \cite[Assumptions \bf{H.tb, H.loc,
  H.sym, H.emb}]{chen15a}.

With the above tight binding Hamiltonion $\mathcal{H}$, 
we can obtain the band energy of the system
\begin{eqnarray}\label{e-band}
E^\Omega(y)=\sum_{s=1}^N f(\varepsilon_s)\varepsilon_s,
\end{eqnarray}
where $(\varepsilon_s)_{s = 1}^N$ are the eigenvalues of $\mathcal{H}(y)$, with
associated eigenvectors $\psi_s$, 
\begin{eqnarray}\label{eigen-H}
\mathcal{H}(y)\psi_s = \varepsilon_s\psi_s\quad s=1,2,\cdots,N,
\end{eqnarray}
% $\psi_s=\{[\psi_s]_{\ell}\}_{1\leq\ell\leq N}$ the associated 
% eigenvectors, 
$f$ the Fermi-Dirac function,
\begin{eqnarray}\label{fermi-dirac}
  f(\varepsilon) = \left( 1+e^{(\varepsilon-\mu)/(k_{\rm B}T)} \right)^{-1},
\end{eqnarray}
$\mu$ a fixed chemical potential, $k_{\rm B}$ Boltzmann's constant, and $T>0$
the temperature of the system.
We do not consider the pairwise repulsive potential, which can be treated purely
classically \cite{chen15a}.

Following \cite{finnis03}, we can distribute the energy to each atomic site 
\begin{eqnarray}\label{E-El}
E^\Omega(y)=\sum_{\ell\in\Omega} E_{\ell}^{\Omega}(y)
\qquad{\rm with}\qquad
E_{\ell}^\Omega(y) := \sum_{s}f(\varepsilon_s)\varepsilon_s
\left|[\psi_s]_{\ell}\right|^2.
\end{eqnarray}

The following theorem \cite[Theorem 3.1 (i)]{chen15a} states the existence of
the limit as $\Omega \uparrow \Lambda$, $\Omega \subset \Lambda$, for some
countable reference domain $\Lambda$. For an infinite body, $\Lambda$, we will
denote this limit site energy by $E_{\ell}$. We will continue to denote the site
energies of subsystems $\Omega \subset \Lambda$ by $E_\ell^\Omega$. When
$\Lambda$ is an infinite reference configuration and $A \subset \R^m$, then we
will also use the short-hand $E_\ell^{A} := E_\ell^{A \cap \Lambda}$.

% by taking the pointwise thermodynamic limit (see ).  
% \hc{
% Throughout this paper, we will use $E_{\ell}$ to denote the 
% thermodynamic limit of the site energy and let 	
% $E^{\Omega}_{\ell}(y)$ denote the site energy of the finite
% subsystem $\Omega\subset\Lambda$.
% }
% The following theorem states the existence of this limit, its
% regularity, locality, and isometry/permutation invariance (see \cite{chen15a}
% for the proof).

\begin{theorem}\label{theorem-thermodynamic-limit}
  Suppose $\Lambda$ is countable, $y \in \mathcal{V}_{\frak{m}}(\Lambda)$ a
  deformation, and $\Omega\subset \Lambda$ a finite subset.
  % Let $\Omega\subset \Lambda$ be any finite subset and
  % $E^{\Omega}_{\ell}(y)$ denote the site energy of the
  % corresponding finite subsystem.  
  Then,
	\begin{itemize}		
		\item[(i)] {\rm (regularity and locality of the site energy)} 
		$E^{\Omega}_{\ell}(y)$ possesses $j$th order partial derivatives with
		$1 \leq j \leq \mathfrak{n}-1$, and there exist 
		constants $C_j$ and $\eta_j$ such that
		\begin{eqnarray}\label{site-locality-tdl}
		\left|\frac{\partial^j E^{\Omega}_{\ell}(y)}{\partial [y(m_1)]_{i_1}
		\cdots\partial [y(m_j)]_{i_j}}\right|
	\leq C_j e^{-\eta_j\sum_{l=1}^j|y(\ell)-y(m_l)|} 
    	\end{eqnarray}
    	with $m_k\in\Omega$ and $1\leq i_k\leq d$ for any $1\leq k\leq j$;
		
		\item[(ii)] {\rm (isometry invariance)}
		$E^{\Omega}_{\ell}(y) = E^{\Omega}_{\ell}(g(y))$ 
		if $g:\R^d\rightarrow\R^d$ is an isometry;
		
		\item[(iii)] (permutation invariance)
		$E^{\Omega}_{\ell}(y) = E^{\mathcal{G}^{-1}(\Omega)}_{\mathcal{G}^{-1}
			(\ell)}(y\circ\mathcal{G})$
		% if $\mathcal{G}:\Omega\rightarrow\Omega$ is a 
		% permutation of $\Omega$;
		for a permutation $\mathcal{G}:\Lambda\rightarrow\Lambda$. 
		
		\item[(iv)] {\rm (thermodynamic limit)}
		$E_{\ell}(y):=\lim_{R\rightarrow\infty} E^{B_R(\ell)}_{\ell}(y)$
		exists and satisfies (i), (ii), (iii) with $\Omega=\Lambda$.
	\end{itemize}
\end{theorem}

For a finite subset $\Omega \subset \Lambda$, we define the (negative) force
\begin{eqnarray}\label{eq:force}
  F^{\Omega}(y):=\nabla E^{\Omega}(y),
  \qquad\text{in component notation,} \qquad
  \big[F_{\ell}^{\Omega}(y)\big]_i = 
  \frac{\partial E^{\Omega}(y)}{\partial [y(\ell)]_i}
  \quad 1\leq i\leq d.
\end{eqnarray} 
Using \eqref{E-El}, we have
\begin{eqnarray}\label{Fl-El}
\big[ F_{\ell}^{\Omega}(y) \big]_i = \sum_{k\in\Omega} 
\frac{\partial E_k^{\Omega}(y)}{\partial [y(\ell)]_i},
\end{eqnarray}
which, together with Theorem \ref{theorem-thermodynamic-limit}, yields the
thermodynamic limit of the force, as well as its regularity, locality, and
isometry/permutation invariance.

\begin{theorem}\label{theorem-thermodynamic-limit-force}
  Suppose the assumptions of Theorem \ref{theorem-thermodynamic-limit} are
  satisfied, then
	\begin{itemize}		
		\item[(i)] {\rm (regularity and locality of the force)} 
		$F^{\Omega}_{\ell}(y)$ possesses $j$th order partial derivatives with
		$1 \leq j \leq \mathfrak{n}-2$, and there exist 
		constants $C_j$ and $\eta_j$ such that
		\begin{eqnarray}\label{site-locality-tdl}
		\left|\frac{\partial^j F^{\Omega}_{\ell}(y)}{\partial [y(m_1)]_{i_1}
			\cdots\partial [y(m_j)]_{i_j}}\right|
		\leq C_j e^{-\eta_j\sum_{l=1}^j|y(\ell)-y(m_l)|} 
		\end{eqnarray}
		with $m_k\in\Omega$ and $1\leq i_k\leq d$ for any $1\leq k\leq j$;
		
		\item[(ii)] {\rm (isometry invariance)}
		$F^{\Omega}_{\ell}\big(Q y + c\big) = Q^{\rm T} F^{\Omega}_{\ell}(y) $
		for all $Q \in {\rm SO}(d), c \in \R^d$;
		
        \item[(iii)] (permutation invariance)
        $F^{\Omega}_{\ell}(y) =
        F^{\mathcal{G}^{-1}(\Omega)}_{\mathcal{G}^{-1}(\ell)}(y\circ\mathcal{G})$
        for a permutation $\mathcal{G}:\Lambda\rightarrow\Lambda$;
		
		\item[(iv)] {\rm (thermodynamic limit)}
		$F_{\ell}(y) := \lim_{R\rightarrow\infty} F_{\ell}^{B_R(\ell)}(y)$
		exists and satisfies (i), (ii), (iii) with $\Omega=\Lambda$.
	\end{itemize}
\end{theorem}

\subsection{Crystalline defects}
\label{sec-defects}

\def\Rcore{R_{\rm DEF}}
\def\Adm{{\rm Adm}}
\def\E{\mathcal{E}}
\def\L{\Lambda}
\def\UsH{\dot{\mathscr{W}}^{1,2}}
\def\Usz{\dot{\mathscr{W}^{\rm c}}}
\def\DD{{\sf D}}
\def\ee{{\sf e}}

\subsubsection{Energy space} 
Let $\Lambda\subset\R^m$ be an infinite reference configuration satisfying
$\L \setminus B_{\Rcore} = (A \Z^m) \setminus B_{\Rcore}$ where
$\Rcore\geq 0$, $A \in {\rm SL}(m)$ and $\Lambda\cap B_{\Rcore}$ is finite.  
For analytical purposes, we assume that there is a regular partition 
$\mathcal{T}_{\Lambda}$ of $\R^m$ into triangles if $m=2$ and tetrahedra 
if $m=3$, whose nodes are the reference sites $\Lambda$ 
(see \ref{sec-interpolation} for interpolations of lattice functions 
on this background mesh).
% \co{[this is problematic; we already need that $\L \sim A \Z^m$??]}

We can decompose the deformation 
\begin{eqnarray}\label{y-u}
y(\ell) = y_0(\ell)+u(\ell) = P\ell + u_0(\ell) 
+ u(\ell) \quad\forall~\ell\in\Lambda,
\end{eqnarray}
where $u_0:\Lambda\rightarrow\mathbb{R}^d$ is a predictor prescribing the
far-field boundary condition, $u:\Lambda\rightarrow\mathbb{R}^d$ is a corrector,
and $P \in \R^{d \times m}$ denotes a macroscopically applied deformation.

If $\ell\in\Lambda$ and
$\ell+\rho\in\Lambda$, then we define the finite difference 
$D_\rho u(\ell) := u(\ell+\rho) - u(\ell)$. 
If $\Rg \subset \Lambda-\ell$, then we
define $D_\Rg u(\ell) := (D_\rho u(\ell))_{\rho\in\Rg}$,
and $Du(\ell) := D_{\Lambda-\ell} u(\ell)$.  
For a stencil $Du(\ell)$ and $\gamma > 0$ we define the (semi-)norms
\begin{eqnarray*}
\big|Du(\ell)\big|_\gamma := \bigg( \sum_{\rho \in \L-\ell} e^{-2\gamma|\rho|} 
\big|D_\rho u(\ell)\big|^2 \bigg)^{1/2}
\quad{\rm and}\quad
\| Du \|_{\ell^2_\gamma} := \bigg( \sum_{\ell \in \L} 
|Du(\ell)|_\gamma^2 \bigg)^{1/2}.
\end{eqnarray*}
An immediate consequence of \eqref{norm-equiv} is that all (semi-)norms
  $\|\cdot\|_{\ell^2_\gamma}, \gamma > 0,$ are equivalent. 

  We can now define the natural function space of finite-energy displacements,
\begin{displaymath}
  \UsH := \big\{ u : \L \to \R^d, \| Du \|_{\ell^2_\gamma} < \infty \big\}.
\end{displaymath}

\subsubsection{Site energy} Let $E_\ell$ denote the site energies we defined
in Theorem \ref{theorem-thermodynamic-limit}(iv). Because they are translation
invariant we can define $V_{\ell} : (\R^m)^{\L-\ell}\rightarrow\R$ by
\begin{eqnarray}
V_{\ell}(Du) := E_{\ell}(Px_0+u) \qquad{\rm with}\quad
x_0:\L\rightarrow\R^d ~~{\rm and}~~ x_0(\ell)=\ell~~\forall~\ell\in\L.
\end{eqnarray}
For a displacement $u$ satisfying $y_0+u\in\mathcal{V}_{\mathfrak{m}}(\Lambda)$,
we can define (formally, for now) the energy-difference functional
\begin{eqnarray}\label{energy-difference}
\mathcal{E}(u) := \sum_{\ell\in\Lambda}\Big(E_{\ell}(y_0+u)-E_{\ell}(y_0)\Big) = \sum_{\ell\in\Lambda}\Big(V_{\ell}(Du_0(\ell)+Du(\ell))-V_{\ell}(Du_0(\ell))\Big).
\end{eqnarray}

For both point defects and dislocations, we can construct predictors $y_0$
  (see \S \ref{sec-p} and \S \ref{sec-d}) such that $\del\E(0) \in (\UsH)^*$. We
  prove in \cite{chenpre_vardef} (see also \cite[Lemma 2.1]{ehrlacher13}) that,
  under this condition, $\E$ is well-defined on the space $\Adm_0$ and in fact
  $\E \in C^{\mathfrak{n}-1}(\Adm_0)$, where
  \begin{displaymath}
    \Adm_{\frak{m}} := \big\{ u \in \UsH, ~
    |y_0(\ell)+u(\ell)-y_0(m)-u(m)| > \frak{m} |\ell-m|
    \quad\forall~  \ell, m \in \L 
    %~\footnote{\hc{Similar to {\bf L}, this condition is not needed
    % for anti-plane models of pure screw dislocations.}} 
    \big\}.
  \end{displaymath}
  In \S \ref{sec-p} and \S \ref{sec-d} we show how the crucial condition
  $\del\E(0) \in (\UsH)^*$ is obtained for, respectively, point defects and
  dislocations. In \S~\ref{sec-uniform} we then present a unified description
  for which we then rigorously state the properties of $\E$ and the associated
  variational problem.

\subsubsection{Point defects}
\label{sec-p}

We make the following standing assumptions for point defects:

\def\asP{{\bf P}\xspace}

\begin{flushleft}
	{\bf P. }   \label{as:asP} 
	$m=d \in \{2, 3 \}$;
	$\exists ~\Rcore>0, A \in {\rm SL}(m)$ such that
	$\Lambda\backslash B_{\Rcore} = (A\mathbb{Z}^m)\backslash B_{\Rcore}$ and
	$\L \cap B_{\Rcore}$ is finite;
	$P={\bf Id}$;
	$u_0=0$;
	$y_0(\ell)=\ell$.
\end{flushleft}

\subsubsection{Dislocations}
\label{sec-d}
The following derivation is not essential to our analysis of QM/MM schemes,
  and can indeed be found in \cite{ehrlacher13}, however, for the sake of
  completeness, we still present enough detail to justify the unified
  formulation in \S~\ref{sec-uniform}.

We consider a model for straight dislocation lines obtained by projecting a 3D
crystal. For a 3D lattice $B\Z^3$ with dislocation direction parallel to $e_3$
and Burgers vector $\burg=(b_1,0,b_3)$, we consider displacements
$W:B\Z^3\rightarrow\R^3$ that are periodic in the direction of the dislocation
direction $e_3$. Thus, we choose a projected reference lattice
$\Lambda:=A\Z^2=\{(\ell_1,\ell_2) ~|~\ell=(\ell_1,\ell_2,\ell_3)\in B\Z^3\}$,
which is again a Bravais lattice. This projection gives rise to a projected 2D
site energy with the additional invariance
\begin{equation}
  \label{eq:projected-site-E-invariance}
  E_\ell(y) = E_\ell(y + z e_3) \qquad \forall z : \Lambda \to \burg_3 \mathbb{Z}
\end{equation}

% For simplicity of notations, we will write 
% $a_{12}=(a_1,a_2)$ for a vector $a\in\R^3$ here and throughout.
% We then identify $W(X)=w(X_{1,2})$ with $w:\Lambda\rightarrow\R^3$.

Let $\hat{x}\in\R^2$ be the position of the dislocation core and
$\Gamma := \{x \in \R^2~|~x_2=\hat{x}_2,~x_1\geq\hat{x}_1\}$ be the ``branch
cut'', with $\hat{x}$ chosen such that $\Gamma\cap\Lambda=\emptyset$.  Following
\cite{ehrlacher13}, we define the far-field predictor $u_0$ by
\begin{eqnarray}\label{predictor-u_0-dislocation}
u_0(x):=\ulin(\xi(x)),
\end{eqnarray}
where $\ulin \in C^\infty(\R^2 \setminus \Gamma; \R^d)$ is the continuum linear
elasticity solution (see \cite{ehrlacher13} for the details)
% satisfying
% \begin{eqnarray}
% \nonumber
% \mathbb{C}_{i\alpha}^{j\beta}\frac{\partial^2 \ulin_i}
% {\partial x_{\alpha}\partial x_{\beta}} =0
% & {\rm in}~\R^2\backslash\Gamma \\[1ex]
% \ulin(x^+)-\ulin(x^-)=-b
% & {\rm for}~ x\in\Gamma\backslash\{\hat{x}\} \\[1ex]
% \nonumber
% \nabla_{e_2}\ulin(x^+)-\nabla_{e_2}\ulin(x^-)=0
% & {\rm for}~ x\in\Gamma\backslash\{\hat{x}\}
% \end{eqnarray}
% with $\mathbb{C}$ being the linearised Cauchy-Born tensor 
% (see e.g. \cite{ehrlacher13}),
%
and % $\xi:\R^2\backslash\Gamma\rightarrow \R^2\backslash\Gamma$ is a map:
\begin{eqnarray}
\xi(x)=x-\burg_{12}\frac{1}{2\pi} 
\eta\left(\frac{|x-\hat{x}|}{\hat{r}}\right)
\arg(x-\hat{x}),
\end{eqnarray}
with $\arg(x)$ denoting the angle in $(0,2\pi)$ between $x$ and
$\burg_{12} = (\burg_1, \burg_2) = (\burg_1, 0)$, and
$\eta\in C^{\infty}(\R)$ with $\eta=0$ in $(-\infty,0]$, $\eta=1$ in
$[1,\infty)$ removes the singularity.

The predictor $y_0 = Px + u_0$ is constructed in such a way that $y_0$
  jumps across $\Gamma$, which encodes the presence of the dislocation. But
  there is an ambiguity in this definition in that we could have equally placed
  the jump into the left half-plane $\{ x_1 \leq \hat{x}_1 \}$. The role of
  $\xi$ in the definition of $u_0$ is that applying a plastic slip across the
  plane $\{ x_2 = \hat{x}_2 \}$ via the definition
  \begin{displaymath}
    y^S(x) := \left\{
      \begin{array}{ll}
        y(\ell), & \ell_2 > \hat{x}_2, \\
        y(\ell-\burg_{12}) - \burg_3 e_3, & \ell_2 < \hat{x}_2
      \end{array}
    \right.
  \end{displaymath}
  achieves exactly this transfer: it leaves the (3D) configuration invariant,
  while generating a new predictor $y_0^S \in C^\infty(\Omega_\Gamma)$ where
  $\Omega_\Gamma = \{ x_1 > \hat{x}_1 + \hat{r} + \burg_1 \}$. Since the
  map $y \mapsto y^S$ represents a relabelling of the atom indices and an
  integer shift in the out-of-plane direction, we can apply
  \eqref{eq:projected-site-E-invariance} and Theorem
  \ref{theorem-thermodynamic-limit} (iii) to obtain
  \begin{equation}
    \label{eq:perm-invariance-site-E-dislocations}
    E_\ell(y) = E_{S^* \ell}(y^S),
  \end{equation}
  where $S$ is the $\ell^2$-orthogonal operator with inverse $S^* = S^{-1}$
  defined by
  \begin{align*}
    Su(\ell) 
    := \left\{ \begin{array}{ll} u(\ell), & \ell_2 > \hat{x}_2, \\
        u(\ell-\burg_{12}), & \ell_2 < \hat{x}_2 
      \end{array} \right.
    \quad {\rm and} \quad
    S^* u(\ell) 
    := \left\{ \begin{array}{ll} u(\ell), & \ell_2 > \hat{x}_2, \\
        u(\ell+\burg_{12}), & \ell_2 < \hat{x}_2.  
      \end{array}\right.
  \end{align*}
  
% the (3D) configuration invariant
%  slip operator
% \begin{align*}
%   S_0 w(\ell) 
%   := \left\{ \begin{array}{ll} w(\ell), & \ell_2 > \hat{x}_2, \\
%                w(\ell-\burg_{12}) - \burg, & \ell_2 < \hat{x}_2 
%              \end{array} \right.
% \end{align*}
% achieves exactly this transfer: if $\hat{r}$ is sufficiently large, then
% $S_0 u_0 \in C^\infty(\Omega_\Gamma)$ where
% $\Omega_\Gamma = \{ x_1 > \hat{x}_1 + \hat{r} + \burg_1 \}$. Further
% $y^S_0 := P x + S_0 u_0$ represents only a relabelling of the atom indices,
% that is, it does not change the configuration. 

  We can translate~\eqref{eq:perm-invariance-site-E-dislocations} to a statement
  about $u_0$ and $V_\ell$. Let $S_0 w(x) = w(x), x_2 > \hat{x}_2$ and
  $S_0 w(x) = w(x-\burg_{12}) - \burg, x_2 < \hat{x}_2$, then we obtain that
  $y_0^S = P x + S_0 u_0$ and $S_0 u_0 \in C^\infty(\Omega_\Gamma)$ and
  $S_0 (u_0 + u) = S_0 u_0 + S u$.
The permutation invariance \eqref{eq:perm-invariance-site-E-dislocations} can
now be rewritten as an invariance of $V_\ell$ under the slip $S_0$:
\begin{eqnarray}\label{slip-invariance}
V\big(D(u_0+u)(\ell)\big) = V\big(e(\ell)+\widetilde{D}u(\ell)\big) \qquad
\forall u \in \Adm_0,~\ell\in\Lambda
\end{eqnarray}
where
\begin{equation}\label{eq:defn_elastic_strain}
e(\ell) := (e_\rho(\ell))_{\rho\in\L-\ell} \quad \text{with} \quad
e_\rho(\ell) 
:= \left\{ \begin{array}{ll}
	S^* D_\rho S_0u_0(\ell), & \ell\in\Omega_{\Gamma}, \\
	D_\rho u_0(\ell), & \text{otherwise,} 
	\end{array} \right.
\end{equation}
and
\begin{equation}\label{eq:defn_Dprime}
\widetilde{D} u(\ell) := 
(\widetilde{D}_\rho u(\ell))_{\rho\in\L-\ell} \quad \text{with} \quad
\widetilde{D}_\rho u(\ell) := 
	\left\{ \begin{array}{ll}
	S^* D_\rho Su(\ell), & \ell \in\Omega_{\Gamma}, \\
	D_\rho u(\ell), & \text{otherwise.}  
	\end{array} \right.
\end{equation}

The following lemma, proven in \cite{chenpre_vardef}, is a straightforward
extension of \cite[Lemma 3.1]{ehrlacher13}.

\begin{lemma}\label{lemma-decay-el}
  If the  predictor $u_0$ is defined by \eqref{predictor-u_0-dislocation} and
  $e(\ell)$ is given by \eqref{eq:defn_elastic_strain}, then
  % for $|\rho|,|\sigma| \leq |\ell|/2$, 
  there exists a constant $C$ such that
  \begin{eqnarray}\label{decay-el}
    % |e(\ell)|_{\gamma} \leq C |\ell|^{-1}
    |e_{\sigma}(\ell)| \leq C |\sigma| \cdot |\ell|^{-1}
    \qquad{\rm and}\qquad
    |D_{\rho} e_{\sigma}(\ell)|_{\gamma} \leq C
    |\rho| \cdot |\sigma| \cdot |\ell|^{-1} .
  \end{eqnarray}
\end{lemma}

We now summarise our standing assumptions for dislocations:

\def\asD{{\bf D}\xspace}

\begin{flushleft}
	{\bf D. }   \label{as:asP} 
	$m=2,~d=3$;
	$\L=A\Z^m$;
	$P(\ell_1,\ell_2)=(\ell_1,\ell_2,\ell_3)$;
	$u_0$ is defined by \eqref{predictor-u_0-dislocation};
	$y_0(\ell)=\ell + u_0(\ell)$.
	
	%\hc{\#\# fit this assumption to that in \cite{chenpre_vardef} later?}
\end{flushleft}

\begin{remark}\label{remark-anti-in-plane}
  One can treat anti-plane models of pure screw dislocations by admitting
  % deformations of the form $y = Px + u_0 + u$ where
  displacements of the form $u_0 = (0, 0, u_{0,3})$ and $u = (0, 0, u_3)$.
  Similarly, one can treat the in-plane models of pure edge dislocations by
  admitting displacements of the form $u_0 = (u_{0,1}, u_{0,2}, 0)$ and
  $u = (u_1, u_2, 0)$ \cite{ehrlacher13}.  For anti-plane models the atoms do
  not accumulate and the condition {\bf L} can be ignored.
\end{remark}

\subsubsection{Unified formulation}
\label{sec-uniform}
In order to consider the point defect and dislocation cases within a unified
notation we introduce the following notation.  Let
\begin{eqnarray*}
	u_0(\ell) := \left\{ \begin{array}{ll}
		0 & {\rm if}~{\bf P} \\
		\eqref{predictor-u_0-dislocation} & {\rm if}~{\bf D}
	\end{array} \right. ,
	\quad
	\ee(\ell) := \left\{ \begin{array}{ll}
		% \L-\ell --- {\bf 0} 
		{\bf 0} & {\rm if}~{\bf P} \\
		e(\ell) & {\rm if}~{\bf D}
	\end{array} \right.
	\quad{\rm and}\quad
	\DD u(\ell) := \left\{ \begin{array}{ll}
		Du(\ell) & {\rm if}~{\bf P} \\
		\widetilde{D}u(\ell) & {\rm if}~{\bf D}
	\end{array} \right. .
\end{eqnarray*}
Using the assumption $u_0=0$ in {\bf P} for point defects and the slip
invariance condition \eqref{slip-invariance} for dislocations, we can rewrite
the energy difference functional \eqref{energy-difference} as
\begin{eqnarray}\label{E-diff}
\E(u) = \sum_{\ell\in \L} 
\Big( V_{\ell}\big( \ee(\ell) + \DD u(\ell) \big) 
- V_{\ell}\big( \ee(\ell) \big) \Big)
\end{eqnarray}
for both point defects and dislocations. The following result is proven in
\cite{chenpre_vardef}, extending an analogous result in \cite{ehrlacher13} which
is restricted to finite-range site energies.

\begin{lemma}
  Suppose that {\bf P} or {\bf D} is satisfied, then $\E$ is well-defined 
  on $\Usz \cap \Adm_0$, where
  \begin{displaymath}
    \Usz := \big\{ u : \L \to \R^d \,\big|\, \exists R > 0 \text{ s.t. }  u = {\rm
    const} \text{ in } \L \setminus B_R \big\},
  \end{displaymath}
  and continuous with respect to the $\UsH$-topology. Therefore, there exists a
  unique continuous extension to $\UsH$ which belongs to
  $C^{\mathfrak{n}-1}(\UsH)$.
\end{lemma}

Having a well-defined energy-difference functional, the equilibrium state can be
determined by solving the variational problem
\begin{equation}\label{eq:appl:variational-problem}
  \bar u \in \arg\min \big\{ \E(u), u \in \Adm_0 \big\},
\end{equation}
where ``$\arg\min$'' is understood in the sense of local minimality.

The next result is an extension of \cite[Theorem 2.3 and 3.5]{ehrlacher13},
which gives the decay estimates for the equilibrium state for point defects and
dislocations (see \cite{chenpre_vardef} for the proof).

\begin{theorem}	\label{th:appl:regularity}
  Suppose that either {\bf P} or {\bf D} is satisfied. If $\bar u \in \Adm_0$ is
  a strongly stable solution to \eqref{eq:appl:variational-problem}, that is,
	\begin{equation}
		\label{eq:appl:strong-stab}
		\exists~ \bar{c} > 0 \text{ s.t. } 
		\big\< \delta^2 \E(\bar u) v, v\big\> \geq \bar{c} 
		\| Dv \|_{\ell^2_\gamma}^2	\qquad \forall v \in\UsH,
	\end{equation}
	then there exists a constant $C > 0$ and $\bar{u}_{\infty}\in\R^d$ such that
%
% for $1\leq j\leq\mathfrak{n}-3$, 
% $\bar{u}$ satisfies the decay
%	\begin{align}\label{decay-estimate}
%		|\DD \bar{u}(\ell)|_\gamma \leq 
%		\left\{ \begin{array}{ll}
%		C (1+|\ell|)^{-m},  & {\rm  if}~{\bf P} \\[1ex]
%		C (1+|\ell|)^{-2}\log|\ell|,   & {\rm  if}~{\bf D}
%		\end{array}  \right.
%		\qquad \forall~\ell\in\L.
%	\end{align}
%	\hc{
%	Moreover, there exists $\bar{u}_{\infty}\in\R^d$ such that
%	\begin{eqnarray}\label{decay-estimate-u-infty}
%          |\bar{u}-\bar{u}_{\infty}|\leq C |\ell| \cdot
% |\DD \bar{u}(\ell)|_\gamma \qquad \forall~\ell\in\L.
%	\end{eqnarray}	
%	}
%
\begin{align}\label{decay-estimate}
|{\sf D}\bar{u}(\ell)|_\gamma &\leq C (1+|\ell|)^{-m} \log^t (2+|\ell|), 
\\[1ex] \label{decay-estimate-uinfty}
|\bar{u}(\ell)-\bar{u}_\infty| &\leq C (1+|\ell|)^{1-m} \log^t (2+|\ell|),
\end{align}
where $t = 0$ for case {\bf P} and $t = 1$ for case {\bf D}.
\end{theorem}

\begin{remark}
  It can be immediately seen that $|D\bar{u}(\ell)|_\gamma$ satisfies the same
  estimate as \eqref{decay-estimate}.
\end{remark}

\section{Energy-mixing}
\label{sec-e-mix-idea}
\setcounter{equation}{0}
\def\LQM{\Lambda^{\rm QM}}
\def\LMM{\Lambda^{\rm MM}}
\def\LFF{\Lambda^{\rm FF}}
\def\Lbuf{\Lambda^{\rm BUF}}
\def\RQM{R_{\rm QM}}
\def\RMM{R_{\rm MM}}
\def\RFF{R_{\rm FF}}
\def\Rbuf{R_{\rm BUF}}
\def\VMM{V^{\rm MM}}
\def\Vb{V^{\rm BUF}_{\#}}
\def\EH{\mathcal{E}^{\rm H}}
\def\uH{\bar{u}^{\rm H}}
\def\wD{\widetilde{D}}
\def\AH{\Adm^{\rm H}_0}
\def\Usx{\mathscr{W}^{\rm H}}

\subsection{Formulation of QM/MM energy mixing}
Following the outline in \S~\ref{sec-e-idealized}, we construct approximations
to the tight binding site energy
% $E_{\ell}(y)\approx E_{\ell}^{\rm MM}(y)$ (or
$V_{\ell}({\bm g})\approx V_{\ell}^{\rm MM}({\bm g})$ for
${\bm g}\in(\R^m)^{\L-\ell}$ by Taylor's expansion, and approximate the energy
difference functional by
\begin{multline}\label{energy-appro} 
	\mathcal{E}(u) 
	% \approx \sum_{\ell\in {\rm QM}} 
	% \big(E_{\ell}(y_0+u)-E_{\ell}(y_0)\big) 
	% + \sum_{\ell\in {\rm MM}} \big(E^{\rm MM}_{\ell}(y_0+u)
	% - E^{\rm MM}_{\ell}(y_0)\big)  
	% \\[1ex]
	\approx \sum_{\ell\in \L^{\rm QM}} 
	\Big( V_{\ell}\big( \ee(\ell) + \DD u(\ell) \big) 
	- V_{\ell}\big( \ee(\ell) \big) \Big)
	+ \sum_{\ell\in \L^{\rm MM}} 
	\Big( V_{\ell}^{\rm MM}\big( \ee(\ell) + \DD u(\ell) \big) 
	- V_{\ell}^{\rm MM}\big( \ee(\ell) \big) \Big).
\end{multline}
Since minimising \eqref{energy-appro} over $u \in \Adm_0$ is an
infinite-dimensional problem, we will also approximate the space of trial
functions.

\subsubsection{Decomposition of $\L$}
\label{sec:decomposition-regions}
We decompose the reference configuration $\L$ into three disjoint sets,
$\Lambda = \LQM\cup \LMM\cup \LFF$, where
$\LQM$ denotes the QM region, $\LMM$ the MM region and
$\LFF$ the {\em far-field} where atom positions will be frozen to those
given by the far-field predictor. Moreover, we define a buffer region
$\Lbuf \subset \LMM$ surrounding $\LQM$ such
that all atoms in $\Lbuf$ are involved in the tight binding calculation when
evaluating the site energies in $\LQM$.

For simplicity, we use balls centred at the point defect or dislocation core to
decompose $\Lambda$, and use parameters $\RQM$, $\RMM$ and $\Rbuf$ to represent
the respective radii, that is,
\begin{displaymath}
	\LQM = B_{\RQM} \cap \L, \qquad \LMM = B_{\RMM} \cap \L \setminus \LQM, 
	\qquad \Lbuf = (B_{\RQM+\Rbuf} \cap \L) \setminus \LQM,
\end{displaymath}
and $\LFF = \L \setminus (\LQM \cup \LMM)$.  See Figure \ref{fig-decomposition}
for a visualisation of this decomposition. 

% cut frm co's slides
\begin{figure}[ht]
	\centering
	% \includegraphics[height=6.0cm]{fig/reference}
	% \hskip 1.5cm
	% \includegraphics[height=6.0cm]{fig/decomposition}
	% \caption{Decomposition of the reference configuration.
	%   Left: the reference configuration. Right: the decomposition and
	%   corresponding approximation parameters.}
	\includegraphics[width=10cm]{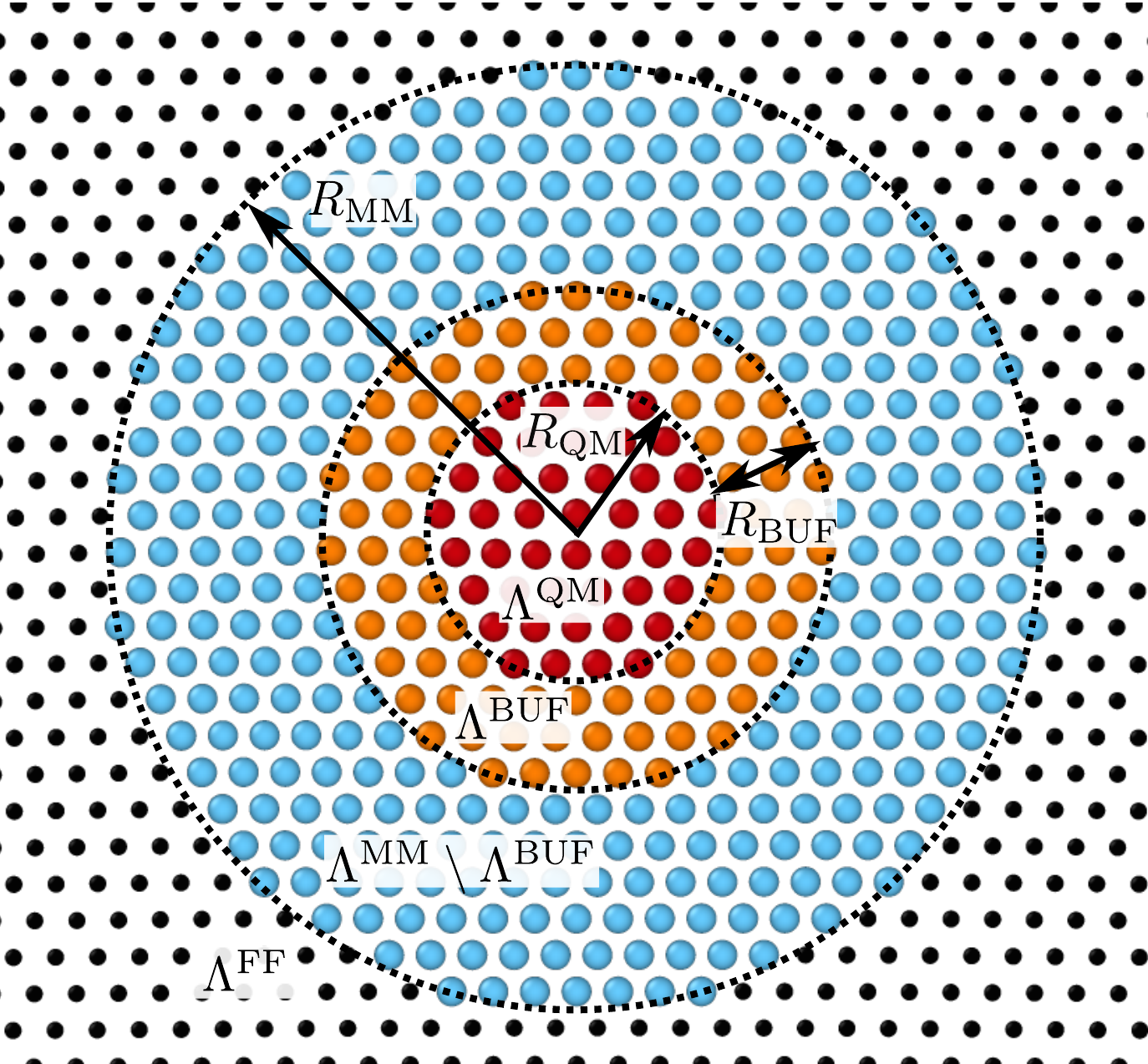}
	\caption{\label{fig-decomposition} Decomposition of a crystal lattice with
		defect (edge dislocation) into QM, MM, buffer and far-field regions,
		according to \S~\ref{sec:decomposition-regions}. }
\end{figure}

\subsubsection{Buffered QM model and site energies}
The site energies in the exact model have infinite range, henve we truncate them
to obtain a computable approximation. To that end, we define
\begin{eqnarray}\label{Vl_buf}
	V_{\ell}^{\rm BUF}\big({\bm g}\big) := 
	\left\{ \begin{array}{ll}
		V_{\ell}^{\LQM\cup\Lbuf}\big({\bm g}\big), & \ell\in\LQM
		\\[1ex]
		V_{\ell}^{B_{\Rbuf}(\ell)}\big({\bm g}\big), & \ell\in\LMM\cup\LFF
	\end{array} \right.,
\end{eqnarray}
where $V_\ell^\Omega(Du(\ell)) := E_\ell^\Omega(Px_0+u)$.

We assume throughout that $R_{\rm QM}>R_{\rm def}+R_{\rm BUF}$. In this case,
Theorem~\ref{theorem-thermodynamic-limit} (ii) (iii) and the assumptions on
$\L$ in {\bf P} and {\bf D} imply that the truncated site potential
\eqref{Vl_buf} is independent of $\ell$ in $\LMM\cup\LFF$.  That is, there
exists $\Vb:(\R^d)^{\mathcal{R}} \rightarrow \R$ such that
\begin{eqnarray}\label{site-hom}
	V_{\ell}^{\rm BUF}\big(Du(\ell)\big) = \Vb\big(D_{\mathcal{R}}u(\ell)\big) 
	\qquad{\rm with}~\mathcal{R} = B_{\Rbuf} \cap (\L \setminus 0),
	\qquad \forall~\ell\in\LMM\cup \LFF.
\end{eqnarray}

\begin{remark}
  We have used the buffer radius parameter $\Rbuf$ for both the buffer
  surrounding the QM region and for the buffer used in the approximate site
  potential $V_\ell^{\rm BUF}$. Although we could choose two separate
  parameters, they affect the error in similar ways, hence for simplicity of
  notation we use only one parameter.
\end{remark}

\subsubsection{QM/MM coupling}
The homogeneity \eqref{site-hom} allows us to construct the MM site potential by
$k$-th order Taylor expansion of $\Vb$ about the far-field lattice state,
\begin{eqnarray}\label{site_MM_k}
	\VMM\big({\bm g}\big) := T_k\Vb\big({\bm g}\big) 
	:= \Vb({\bf 0}) + \sum_{j=1}^k \frac{1}{j!} 
	\delta^j \Vb({\bf 0})\left[{\bm g}^{\otimes j}\right]
	\quad{\rm with}~~k\geq 2.
\end{eqnarray}

With the definitions \eqref{Vl_buf} and \eqref{site_MM_k} we can now specify the
QM/MM energy-mixing scheme
\begin{eqnarray}\label{problem-e-mix}
	\bar{u}^{\rm H} \in \arg\min\big\{ \mathcal{E}^{\rm H}(u) ~\lvert~ u\in \AH \big\},
\end{eqnarray}
with the QM/MM hybrid energy difference functional
\begin{equation}\label{E-H}
  \begin{split}
    \E^{\rm H}(u) 
    &= \sum_{\ell\in \LQM} 
      \Big( V^{\rm BUF}_{\ell}\big(\ee(\ell) + \DD u(\ell)\big) 
      - V^{\rm BUF}_{\ell}\big(\ee(\ell)\big) \Big) \\
    & \qquad + \sum_{\ell\in \LMM\cup\LFF} 
    \Big( \VMM\big(\ee(\ell) + \DD u(\ell)\big) - \VMM\big(\ee(\ell)\big) \Big)
  \end{split}
\end{equation}
and admissible set 
\begin{eqnarray}\label{e-mix-space}
	\Adm_{\frak{m}}^{\rm H} := \Adm_{\frak{m}} \cap \Usx \qquad \text{where} \qquad
	\Usx := \left\{ u \in \Usz  ~\lvert~ u=0~{\rm in}~\Lambda^{\rm FF} \right\} .
\end{eqnarray}

Using same arguments as those in \cite{chenpre_vardef,ehrlacher13}, we have that
$\EH\in C^{\mathfrak{n}-1}\big(\AH\big)$. Note, in particular, that the sum
over $\LMM \cup \LFF$ in the definition of $\E^{\rm H}$ is in fact finite,
since $V^{\rm MM}$ has a finite range of interaction.

\subsection{Error estimates}
\label{sec-e-analysis}
The QM/MM energy mixing scheme \eqref{problem-e-mix} satisfies the following
approximation error estimate. The main steps of the proof are presented below,
but some technical details are given in the appendix.

\begin{theorem}\label{theorem-e-mix}
	Suppose that either assumption {\bf P} or {\bf D} is satisfied and that
	$\bar{u}$ is a strongly stable solution of
	\eqref{eq:appl:variational-problem}.
	
	If, in the definition of $\E^{\rm H}$ in \eqref{E-H}, $\VMM$ is the $k$-th
	order expansion in \eqref{site_MM_k} and $\mathfrak{n}\geq k+2$, then there
	exist constants $C,~\kappa, ~c_{\rm BUF}^{\rm QM},~c_{\rm BUF}^{\rm MM}$ such
	that, for $\RQM$ sufficiently large and for
	\begin{displaymath}
		\Rbuf \geq \max\{c_{\rm BUF}^{\rm QM}\log\RQM,c_{\rm BUF}^{\rm MM}\log\log\RMM\}
		% \leq {\textstyle \frac13} \RQM,
	\end{displaymath}
	there exists a strongly stable solution $\uH$ of \eqref{problem-e-mix}
	satisfying 
	\begin{align}
		\label{error-u-e-mix}
		\|D\bar{u}-D\uH\|_{\ell^2_\gamma} &\leq 
		C\left( \RQM^{-\alpha} + \RMM^{-m/2} \log^t\RMM + e^{-\kappa \Rbuf} \right),  
                                                    \quad \text{and} 
		\\[1ex]
		\label{error-e-e-mix}
		|\E(\bar{u})-\EH(\uH)| &\leq 	
		C\left( R_{\rm QM}^{-\alpha-m/2} \log^t\RQM + R_{\rm MM}^{-m} \log^{2t}\RMM 
		+ e^{-\kappa \Rbuf} \right), \\[1ex]
		\nonumber
		\text{where} \qquad & \begin{cases}
			\alpha = (2k-1)m/2, & \text{if {\bf P}}, \\
			\alpha = k-1, & \text{if {\bf D}}.
		\end{cases}
	\end{align}
\end{theorem}

\begin{proof}
	{\it 1. Quasi-best approximation: } 
	Following \cite[Lemma 7.3]{ehrlacher13}, we can construct
	$T^{\rm H}\bar{u}\in\AH$ by
	% For $R > 0$ we define $T_R : (\R^d)^\L \to \Usz(\L)$ by
	\begin{equation}\label{eq:atm:TR}
		T^{\rm H} \bar{u}(\ell) := \eta\big( \ell / \RMM \big) 
		\big(\bar{u}(\ell) - \bar{u}_{\infty} - a_{\RMM} \big)  ,
	\end{equation}
	where $\bar{u}_{\infty}$ is given in Theorem \ref{th:appl:regularity},
	$a_{\RMM} := \mint_{B_{5\RMM/6} \backslash B_{4\RMM/6}}	
	\big(I\bar{u}(x) - \bar{u}_{\infty} \big) \dd x$
	with $I\bar{u}$ defined in \S \ref{sec-interpolation},
	and $\eta \in C^1(\R^m)$ is a cut-off function satisfying
	$\eta(x) = 1$ for $|x| \leq 4/6$ and $\eta(x) = 0$ for $|x| \geq 5/6$.
	Then, for any $\gamma > 0$ and for $\RMM$ sufficiently large,
	we have from the decay estimates in Theorem \ref{th:appl:regularity} that
	\begin{align}
		\label{proof-4-1-1}
		&\|DT^{\rm H}\bar{u}-D\bar{u}\|_{\ell^2_\gamma} 
		\leq C \|D\bar{u}\|_{\ell^2_\gamma(\Lambda\backslash B_{\RMM/2})}
		\qquad{\rm and} 
		\\[1ex] \label{D-DD}
		& 
		|\DD T^{\rm H}\bar{u}(\ell)|_{\gamma} \leq C  
		(1+|\ell|)^{-m} \log^t (2+|\ell|)  \quad \forall~\ell\in\L.
	\end{align}

	Let $r>0$ be such that $B_r(\bar{u}) \subset \Adm_{\frak{m}}$ for some
	$\frak{m} > 0$. We have from Theorem \ref{th:appl:regularity} that, for $\RMM$
	sufficiently large, $T^{\rm H}\bar{u} \in B_{r/2}(\bar{u})$ and hence
	$B_{r/2}(T^{\rm H}\bar{u}) \subset \Adm_{\frak{m}}$. Since $\E\in C^3(\Adm_0)$,
	$\delta\E$ and $\delta^2\E$ are Lipschitz continuous in $B_r(\bar{u})$ with
	uniform Lipschitz constants $L_1$ and $L_2$, i.e.,
	\begin{align}
		\label{proof-4-1-2}
		\|\delta\E(\bar{u})-\delta\E(T^{\rm H}\bar{u})\| 
		&\leq L_1\|D\bar{u}-DT^{\rm H}\bar{u}\|_{\ell^2_\gamma}
		\leq CL_1\|D\bar{u}\|_{\ell^2_\gamma(\L\backslash B_{\RMM/2})}
		\qquad \text{and} \\
		\label{proof-4-1-3}
		\|\delta^2\E(\bar{u})-\delta^2\E(T^{\rm H}\bar{u})\| 
		&\leq L_2\|D\bar{u}-DT^{\rm H}\bar{u}\|_{\ell^2_\gamma}
		\leq CL_2\|D\bar{u}\|_{\ell^2_\gamma(\L\backslash B_{\RMM/2})}.
	\end{align}
	
	{\it 2. Stability: } Since $\bar{u}$ is strongly stable, there exists
	$\bar{c} > 0$ such that
	$\< \ddel \E^{\rm H}(\bar{u}) v, v \> \geq \bar{c} \|Dv \|_{\ell^2_\gamma}^2$.
	For any $v\in\Usx$, we have
	\begin{eqnarray}\label{proof-s-T}
		\nonumber
		&& \big\< \delta^2\EH(T^{\rm H}\bar{u}) v, v \big\> 
		-\big\< \delta^2\E(\bar{u}) v, v \big\>  
		\\[1ex]
		\nonumber
		&=& \big\< \big(\delta^2\EH(T^{\rm H}\bar{u}) 
		- \delta^2\E(T^{\rm H}\bar{u})\big) v, v \big\> 
		+ \big\< \big(\delta^2\E(T^{\rm H}\bar{u}) 
		- \delta^2\E(\bar{u})\big) v, v \big\> 
		\\[1ex]
		\nonumber
		&= & \sum_{\ell\in\L} 
		\left\< \Big(\delta^2 V_{\ell}^{\rm BUF}\big(\ee(\ell)+\DD T^{\rm H}\bar{u}(\ell)\big) 
		- \delta^2 V_{\ell}\big(\ee(\ell)+\DD T^{\rm H}\bar{u}(\ell)\big) \Big) 
		\DD v(\ell) , \DD v(\ell) \right\> 
		\\[1ex]
		\nonumber
		&& + \sum_{\ell\in\LMM\cup\LFF} 
		\left\< \Big( \delta^2 \VMM\big(\ee(\ell)+\DD T^{\rm H}\bar{u}(\ell)\big) 
		- \delta^2 V_{\ell}^{\rm BUF}\big(\ee(\ell)+\DD T^{\rm H}\bar{u}(\ell)\big) \Big) 
		\DD v(\ell) , \DD v(\ell) \right\> 
		\\[1ex]
		\nonumber
		&& + \big\< \big(\delta^2\E(T^{\rm H}\bar{u}) 
		- \delta^2\E(\bar{u}) \big)v, v \big\> 
		\\[1ex]
		&=:&  Q_1 + Q_2 + Q_3 .
	\end{eqnarray}
	Using the estimate \eqref{Vl-buf-err}, we have
	\begin{eqnarray}\label{proof-s-T1}
		|Q_1|\leq C\sum_{\ell\in\L} e^{-\eta\Rbuf} |\DD v(\ell)|_{\gamma}^2
		\leq Ce^{-\eta\Rbuf} \|Dv\|^2_{\ell^2_{\gamma}} .
	\end{eqnarray}
	Taylor's expansion \eqref{site_MM_k} yields
	\begin{multline}\label{proof-s-T2}
		|Q_2| = \sum_{\ell\in\LMM\cup\LFF} 
		\left\< \Big( \delta^2 T_k\Vb\big(\ee(\ell)+\DD T^{\rm H}\bar{u}(\ell)\big) 
		- \delta^2 V_{\#}^{\rm BUF}\big(\ee(\ell)+\DD T^{\rm H}\bar{u}(\ell)\big) \big) 
		\DD v(\ell) , \DD v(\ell) \right\>  
		\\[1ex]
		\leq C\sum_{\ell\in\LMM\cup\LFF}
		|\ee(\ell)+\DD T^{\rm H}\bar{u}(\ell)|_{\gamma}^{k-1} |\DD v(\ell)|^2_{\gamma}
		~\leq~ C\|\ee+\DD T^{\rm H}\bar{u}\|^{k-1}_{\ell^{\infty}_{\gamma}(\LMM\cup\LFF)}
		\|Dv\|^2_{\ell^2_{\gamma}}.
	\end{multline}
	
	The Lipschitz continuity \eqref{proof-4-1-3} implies
	\begin{eqnarray}\label{proof-s-T3}
		|Q_3| \leq CL_2 
		\|D\bar{u}\|_{\ell^2_\gamma(\L\backslash B_{\RMM/2})} 
		\|Dv\|^2_{\ell^2_{\gamma}} .
	\end{eqnarray}
	Using \eqref{D-DD}, \eqref{proof-s-T}, \eqref{proof-s-T1}, \eqref{proof-s-T2},
	\eqref{proof-s-T3}, the decay estimates in Lemma \ref{lemma-decay-el},
	Theorem \ref{th:appl:regularity}, and
	the fact that $\bar{u}$ is a strongly stable solution, we have that for
	sufficiently large $\RQM$ and $\Rbuf$ (note that $\RMM \geq \RQM$),
	\begin{eqnarray}\label{e-mix-stability}
		\big\< \delta^2\EH(T^{\rm H}\bar{u}) v, v \big\>  \geq \frac{\bar{c}}{2} 
		\| Dv \|_{\ell^2_\gamma}^2.
	\end{eqnarray}

	{\it 3. Consistency: } We estimate the consistency error, for any $v\in\Usx$, by
	\begin{eqnarray}\label{proof-c-T}
		\nonumber
		&& \hspace{-1.5cm} \big\< \delta\EH(T^{\rm H}\bar{u}) , v \big\> 
		\\[1ex]
		\nonumber
		&=& \big\< \delta\EH(T^{\rm H}\bar{u}) - \delta\E(T^{\rm H}\bar{u}) , v \big\> 
		+ \big\< \delta\E(T^{\rm H}\bar{u}) - \delta\E(\bar{u}) , v \big\> 
		\\[1ex]
		\nonumber
		&= & \sum_{\ell\in\L} 
		\big\< \delta V_{\ell}^{\rm BUF}\big(\ee(\ell)+\DD T^{\rm H}\bar{u}(\ell)\big) 
		- \delta V_{\ell}\big(\ee(\ell)+\DD T^{\rm H}\bar{u}(\ell)\big) , \DD v(\ell) \big\> 
		\\[1ex]
		\nonumber
		&& + \sum_{\ell\in\LMM\cup\LFF} 
		\big\< \delta \VMM\big(\ee(\ell)+\DD T^{\rm H}\bar{u}(\ell)\big) 
		- \delta V_{\ell}^{\rm BUF}\big(\ee(\ell)+\DD T^{\rm H}\bar{u}(\ell)\big) , 
		\DD v(\ell) \big\> 
		\\[1ex]
		\nonumber
		&& + \big\< \delta\E(T^{\rm H}\bar{u}) - \delta\E(\bar{u}) , v \big\> 
		\\[1ex]
		&:=&  T_1 + T_2 + T_3 .
	\end{eqnarray}
	The term $T_1$ can be estimated by
	\begin{eqnarray}\label{proof-c-T1}
		|T_1| \leq C e^{-\kappa\Rbuf}\|Dv\|_{\ell^2_{\gamma}},
	\end{eqnarray}
	with some constant $\kappa$; a detailed proof of this assertion is presented in
	\ref{sec-proof-c-T1}.
	
	To estimate $T_2$, we have from \eqref{site_MM_k} that
	\begin{eqnarray}\label{proof-c-T2}
		\nonumber
		|T_2| &=& \sum_{\ell\in\LMM\cup\LFF} 
		\big\< \delta T_k\Vb\big(\ee(\ell)+\DD T^{\rm H}\bar{u}(\ell)\big) 
		- \delta V_{\#}^{\rm BUF}\big(\ee(\ell)+\DD T^{\rm H}\bar{u}(\ell)\big) , 
		\DD v(\ell) \big\>  
		\\[1ex] \nonumber
		&\leq& C\sum_{\ell\in\LMM\cup\LFF}
		|\ee(\ell)+\DD T^{\rm H}\bar{u}(\ell)|_{\gamma}^{k} |\DD v(\ell)|_{\gamma}
		\\[1ex]
		&\leq& C\|\ee+\DD T^{\rm H}\bar{u}\|^{k}_{\ell^{2k}_{\gamma}(\LMM\cup\LFF)}
		\|Dv\|_{\ell^2_{\gamma}}
	\end{eqnarray}	
	
	Further, using \eqref{proof-4-1-2} we can estimate $T_3$ by
	\begin{eqnarray}\label{proof-c-T3}
		|T_3| \leq CL_1 
		\|D\bar{u}\|_{\ell^2_\gamma(\L\backslash B_{\RMM/2})} 
		\|Dv\|_{\ell^2_{\gamma}} .
	\end{eqnarray}
	Taking into accounts \eqref{proof-c-T}, \eqref{proof-c-T1}, 
	\eqref{proof-c-T2} and \eqref{proof-c-T3}, we have
	\begin{align}\label{e-mix-consistency}
		\big\< \delta\EH(T^{\rm H}\bar{u}) , v \big\> \leq C
		\left( e^{-\kappa\Rbuf} + \|\ee+\DD T^{\rm H} 
		\bar{u}\|^{k}_{\ell^{2k}_{\gamma}(\LMM\cup\LFF)}
		+ \|D\bar{u}\|_{\ell^{2}_\gamma(\L\backslash B_{\RMM/2})} 
		\right) \|Dv\|_{\ell^2_{\gamma}} .
	\end{align}
	
	If {\bf P} is satisfied, then we can obtain the estimates for 
	point defects by substituting $\ee(\ell)={\bf 0}$ and
	$|\DD \bar{u}(\ell)|_\gamma \leq C (1+|\ell|)^{-m}$
	into \eqref{e-mix-consistency}:
	\begin{eqnarray}\label{e-mix-consistency-P}
		\left| \big\< \delta\EH(T^{\rm H}\bar{u}) , v \big\> \right| \leq C
		\left( \RQM^{-(2k-1)m/2} + \RMM^{-m/2} + e^{-\kappa \Rbuf} \right)
		\|Dv\|_{\ell^2_{\gamma}} .
	\end{eqnarray}
	If {\bf D} is satisfied, then we can obtain the estimates for
	dislocations by substituting $\ee(\ell)=e(\ell)$,
	$|e(\ell)|_{\gamma} \leq C|\ell|^{-1}$, and
	$|\DD \bar{u}(\ell)|_\gamma \leq C (1+|\ell|)^{-2}\log(2+|\ell|) $
	into \eqref{e-mix-consistency}:
	\begin{eqnarray}\label{e-mix-consistency-D}
		\left| \big\< \delta\EH(T^{\rm H}\bar{u}) , v \big\> \right|  \leq C
		\left( \RQM^{-k+1} + \RMM^{-1}\log\RMM + e^{-\kappa \Rbuf} \right)
		\|Dv\|_{\ell^2_{\gamma}} .
	\end{eqnarray}

	{\it 4. Application of inverse function theorem: } With the stability
	\eqref{e-mix-stability} and consistency \eqref{e-mix-consistency}, we can apply
	the inverse function theorem \cite[Lemma 2.2]{ortner11} to obtain, for
	$\RQM, \Rbuf$ sufficiently large, the existence of a solution $\uH$ to
	\eqref{problem-e-mix}, and the estimate
	\begin{eqnarray}\label{err-e-mix-u}
		\|D\uH-DT^{\rm  H}\bar{u}\|_{\ell^2_{\gamma}} \leq C \left\{
		\begin{array}{ll}
			\RQM^{-(2k-1)m/2} + \RMM^{-m/2} + e^{-\kappa \Rbuf} &  {\rm if~\bf P}
			\\[1ex]
			\RQM^{-k+1} + \RMM^{-1}\log\RMM + e^{-\kappa \Rbuf} & {\rm if~\bf D}
		\end{array}   \right. ,
	\end{eqnarray}
	which together with \eqref{proof-4-1-1} completes the proof of \eqref{error-u-e-mix}.
	The error estimate, together with the stability estimate \eqref{e-mix-stability}
	in particular imply that, for $\RQM, \Rbuf$ sufficiently large, $\uH$ is
	strongly stable.
	
	{\it 5. Error in the energy: } Next, we estimate the error in the energy
	difference functional.  From $\E\in C^2(\Adm_0)$ we have that
	\begin{multline}\label{proof-e-1}
		\big| \E(T^{\rm H}\bar{u}) - \E(\bar{u}) \big| 
		= \Big| \int_0^1 \big\< \delta\E \big( 	(1-s)
		\bar{u} + s T^{\rm H}\bar{u}\big), T^{\rm H}\bar{u}
		- \bar{u} \big\> \dd s \Big| \\
		= \Big| \int_0^1 \big\< \del\E \big( (1-s) \bar{u} 
		+ s T^{\rm H}\bar{u}\big) -
		\del\E(\bar{u}), T^{\rm H}\bar{u} - \bar{u} \big\> \dd s \Big| 
		\leq C \| DT^{\rm H}\bar{u} - D\bar{u} \|_{\ell^2_\gamma}^2  ~~
	\end{multline}
	and from  $\EH\in C^2(\AH)$ that
	\begin{eqnarray}\label{proof-e-1-}
		\big| \EH(\uH) - \EH(T^{\rm H}\bar{u}) \big| 
		\leq C \| D\uH - DT^{\rm H}\bar{u} \|_{\ell^2_\gamma}^2 .
	\end{eqnarray}
	
	Denoting $g(\ell)=\DD T^{\rm H}\bar{u}(\ell)$ and suppressing the argument
	$(\ell)$ in $g(\ell)$ and $\ee(\ell)$, we have
	\begin{eqnarray}\label{proof-e-2}
		\nonumber
		&& \hspace{-2cm} |\E(T^{\rm H}\bar{u})-\EH(T^{\rm H}\bar{u})| 
		\\[1ex] \nonumber
		&=& \sum_{\ell\in \Lambda} 
		\left( V_{\ell}(g+\ee) - V_{\ell}(\ee)
		- V^{\rm BUF}_{\ell}(g+\ee) + V^{\rm BUF}_{\ell}(\ee) \right) 
		\\ \nonumber
		&& +\sum_{\ell\in \LMM\cup\LFF} 
		\left( \Vb(g+\ee) - \Vb(\ee)
		- T_k\Vb(g+\ee) + T_k\Vb(\ee) \right) 
		\\[1ex]
		&:=& S_1+S_2 ,
	\end{eqnarray}
	where $S_1$ is estimated in \ref{sec-proof-c-T1} by
	\begin{eqnarray}\label{proof-e-3}
		|S_1| \leq C e^{-\kappa\Rbuf} ,
	\end{eqnarray}
	and $S_2$ is estimated by
	\begin{align}\label{proof-e-4}
		\nonumber
		|S_2| &\leq 
		\sum_{\ell\in \LMM\cup\LFF} \left(
		\frac{1}{(k+1)!} \big( \delta^{k+1}\Vb(0)(g+\ee)^{\otimes k+1}  
		- \delta^{k+1}\Vb(0)\ee^{\otimes k+1} \big) \right.
		\\[1ex] \nonumber
		& \qquad  \left. + C\big( |g+\ee|_{\gamma}^{k+2} 
		+ |\ee|_{\gamma}^{k+2} \big) \right)
		\\[1ex] \nonumber
		&\leq C\sum_{\ell\in \LMM\cup\LFF} \left(
		|g|_{\gamma} (|g+\ee|_{\gamma}^k + |\ee|_{\gamma}^k) 
		+ |g+\ee|_{\gamma}^{k+2} + |\ee|_{\gamma}^{k+2} \right)
		\\[1ex] \nonumber
		&\leq C \left( \| |\DD T^{\rm H}\bar{u}(\ell)|_{\gamma} \cdot 
		|\ee(\ell)+\DD T^{\rm H}\bar{u}(\ell)|_{\gamma}^k 
		\|_{\ell^1_{\gamma}(\LMM\cup\LFF)} 
		+  \||\ee(\ell)+\DD T^{\rm H}\bar{u}(\ell)|_{\gamma}^{k+2} 
		\|_{\ell^1_{\gamma}(\LMM\cup\LFF)}  
		\right)
		\\[1ex]
		&\leq C \left\{ \begin{array}{ll}
			\RQM^{-km}  &  {\rm if~\bf P}
			\\[1ex]
			\RQM^{-k}\log\RQM  & {\rm if~\bf D}
		\end{array}   \right. .
	\end{align}
	
	Taking \eqref{err-e-mix-u}, \eqref{proof-e-1}, \eqref{proof-e-1-}, 
	\eqref{proof-e-2}, \eqref{proof-e-3} and
	\eqref{proof-e-4} into accounts, we have
	\begin{eqnarray}\label{err-e-mix-e}
		\nonumber
		|\EH(\uH) - \E(\bar{u})| &\leq& 
		|\EH(\uH) - \EH(T^{\rm H}\bar{u})| 
		+ |\EH(T^{\rm H}\bar{u}) - \E(T^{\rm H}\bar{u})| 
		+ |\E(T^{\rm H}\bar{u}) - \E(\bar{u})|
		\\[1ex]
		&\leq& C
		%\left( \| D\uH - D\bar{u} \|_{\ell^2_\gamma}^2+ \| |\DD \uH(\ell)|\cdot 
		% |\ee(\ell)+\DD \uH(\ell)|^k \|_{\ell^1_{\gamma}(\LMM\cup\LFF)} + e^{-\kappa\Rbuf} \right)
		\left\{ \begin{array}{ll}
			\RQM^{-km} + \RMM^{-m} + e^{-\kappa \Rbuf} &  {\rm if~\bf P}
			\\[1ex]
			\RQM^{-k}\log\RQM + \RMM^{-2}\log^2\RMM + e^{-\kappa \Rbuf} & {\rm if~\bf D}
		\end{array}   \right. ,
	\end{eqnarray}
	which completes the proof of \eqref{error-e-e-mix}.
\end{proof}

\section{Force-mixing}
\label{sec-f-mix-idea}
\setcounter{equation}{0}

\def\Fb{\F^{\rm BUF}_{\#}}
\def\FMM{\F^{\rm MM}_{\#}}
\def\FH{\F^{\rm H}}
\def\Fa{\widetilde{\F}}
\def\Fab{\widetilde{\F}^{\rm BUF}_{\#}}
\def\Ea{\widetilde{\E}}
\def\LI{\L^{\rm I}}

\subsection{Formulation of QM/MM force mixing}
\label{sec-formulation-f-mix}
To construct a force-based QM/MM coupling scheme, we follow the idea in
\S~\ref{sec-f-idealized}. In the MM region we construct an approximation to the
tight binding force $F_{\ell}(y)\approx F_{\ell}^{\rm MM}(y)$ by Taylor's
expansion in order to ensure a good match between the QM and MM models.

Our starting point, instead of the energy minimisation formulation
\eqref{eq:appl:variational-problem}, is the force-equilibrium formulation:
\begin{eqnarray}\label{eq:problem-force}
	{\rm Find}~ \bar{u} \in \Adm_0, ~~ {\rm s.t.} \quad
	F_{\ell}(y_0 + \bar{u}) = 0 \qquad \forall~\ell\in\Lambda.
\end{eqnarray}
where
\begin{eqnarray}\label{eq:force-Du}
	F_{\ell}(y_0+u) = 
	\sum_{\rho\in\ell-\L}
	V_{\ell-\rho,\rho}\big(Du_0(\ell-\rho) + Du(\ell-\rho)\big)
	- \sum_{\rho\in\L-\ell}  V_{\ell,\rho}\big(Du_0(\ell) + Du(\ell)\big).
\end{eqnarray}
We have from \eqref{energy-difference} and Theorem
\ref{theorem-thermodynamic-limit-force} (iv) that $F(y_0+u) = \nabla \E(u)$,
hence any solution of \eqref{eq:appl:variational-problem} also solves
\eqref{eq:force-Du}.

To simplify the notation in the construction of the QM/MM scheme we define
\begin{displaymath}
	\F_\ell^\Omega(u) := F_\ell^\Omega(y_0+u) \qquad \text{and} \qquad
	\Fa_\ell^\Omega(w) := F_\ell^\Omega(Px_0 + w),
\end{displaymath}
and we remark that 
\begin{equation}\label{eq-force-Du-Omega}
	\Fa_\ell^\Omega(u_0+u) = \F_\ell^\Omega(u) = 
	\sum_{\rho\in \ell-\Omega} V^{\Omega}_{\ell-\rho,\rho}\big(Du_0(\ell-\rho)
	+ Du(\ell-\rho)\big) - \sum_{\rho\in\Omega-\ell} 
	V^{\Omega}_{\ell,\rho}\big(Du_0(\ell) + Du(\ell)\big) .
\end{equation}

We decompose the reference configuration into $\LQM$, $\LMM$, $\LFF$, $\Lbuf$,
in the same way as in \S~\ref{sec-e-mix-idea}. To obtain computable forces we
then truncate the force of the infinite lattice,
\begin{eqnarray}\label{Fl_buf}
	\F_{\ell}^{\rm BUF}(u) := 
	\left\{ \begin{array}{ll}
		\F_{\ell}^{\LQM\cup\Lbuf}(u), & \ell\in\LQM,
		\\[1ex]
		\F_{\ell}^{B_{\Rbuf}(\ell)}(u), & \ell\in\LMM\cup\LFF.
	\end{array} \right. 
\end{eqnarray}
If $R_{\rm QM}>\Rcore+R_{\rm BUF}$, then Theorem
\ref{theorem-thermodynamic-limit-force} (ii) (iii) and the assumptions on
$\Lambda$ in {\bf P} and {\bf D} imply that the truncated force operator
$\Fa_{\ell}^{B_{\Rbuf}(\ell)}$ is independent of $\ell$ in $\LMM\cup\LFF$.  That
is, there exists $\Fb:(\R^m)^{\mathcal{R}} \rightarrow \R$, where
  $\mathcal{R} = (A \Z^m) \cap B_{\Rbuf}$ such that
\begin{eqnarray}\label{force-hom}
	\Fa_{\ell}^{B_{\Rbuf}(\ell)}(v) = 
	\Fb\left(v(\cdot-\ell)|_{B_{\Rbuf}}\right)
	\qquad \forall~\ell\in\LMM\cup \LFF.
\end{eqnarray}
%
% In what follows, we will suppress the argument $(\ell)$
% and write $\Fb(u)$ for simplicity of notations.
We now define the MM force to be the $k$-th order Taylor expansion of $\Fb$,
\begin{eqnarray}\label{force_MM_k}
	% &\hspace{-0.5cm} 
	\FMM(w) ~:=~ T_k\Fb(w)
	~:=~ \Fb(0) + \sum_{j=1}^k \frac{1}{j!} 
	\delta^j \Fb(0)\left[w^{\otimes j}\right]
	\qquad{\rm with}~~k\geq 1 .
\end{eqnarray}
% with $k\geq 1$. 
We remark that the zeroth-order term in the expansion
vanishes since the reference lattice is an equilibrium.

We have the following force-mixing QM/MM coupling model:
\begin{eqnarray}\label{problem-f-mix}
	{\rm Find}~\uH\in \AH, ~~ {\rm s.t.} \quad
	\F_{\ell}^{\rm H}\big(\bar{u}^{\rm H}\big) = 0
	\qquad \forall~\ell\in\LQM\cup\LMM
\end{eqnarray}
with the hybrid force
\begin{eqnarray}\label{F-H}
	\F_{\ell}^{\rm H}(u) = 
	\left\{ \begin{array}{ll}
		\F^{\rm BUF}_{\ell}(u) & \ell\in\LQM
		\\[1ex]
		\FMM\left( \big(u_0(\cdot-\ell)+u(\cdot-\ell)\big)
		|_{B_{\Rbuf}(\ell)} \right) & \ell\in\LMM  %\cup\LFF
	\end{array} \right. .
\end{eqnarray}
We emphasize that $\FH$ is not a gradient of any energy functional.  For
$v:\L\rightarrow\R$, we will use the notation
$\big\<\FH(u),v\big\>:=\sum_{\ell\in\LQM\cup\LMM}\FH_{\ell}(u)\cdot v(\ell)$ in
our analysis.

\subsection{Error estimates}
\label{sec-f-analysis}

\begin{theorem}\label{theorem-f-mix}
	Suppose that either assumption {\bf P} or {\bf D} is satisfied and that
	$\bar{u}$ is a strongly stable solution of \eqref{eq:problem-force}.
	
	Suppose that, in the definition of $\FH$ in \eqref{F-H}, $\FMM$ is the $k$-th
	order expansion in \eqref{force_MM_k} and $\mathfrak{n}\geq k+3$. Then, for
	any given MM region growth constant $C_{\rm QM}^{\rm MM} > 0$, there exist
	constants $C,~\kappa, ~c_{\rm BUF}^{\rm QM},~c_{\rm BUF}^{\rm MM}$ such that,
	if $\RQM$ is sufficiently large, while $\RQM, \Rbuf, \RMM$ maintain the bounds
	\begin{displaymath}
		\log \frac{\RMM}{\RQM} \leq C_{\rm QM}^{\rm MM} \qquad
		\text{and} \qquad
		\Rbuf \geq \max\{c_{\rm BUF}^{\rm QM}\log\RMM, c_{\rm BUF}^{\rm MM}\log\log\RMM\},
	\end{displaymath}
	there exists a strongly stable solution $\uH$ of \eqref{problem-f-mix}
	satisfying
	\begin{align}\label{error-u-f-mix}
		\|D\bar{u}-D\uH\|_{\ell^2_\gamma}
		&\leq 
		C\left( \RQM^{-\alpha} \log\RMM + \RMM^{-m/2}\log^t\RMM
		+ e^{-\kappa \Rbuf} \right),   
		\\ \nonumber
		\text{where} 
		\qquad
		& 
		\begin{cases}
			\alpha=(2k+1)m/2, & \text{if {\bf P}}, \\
			\alpha=k, & \text{if {\bf D}}.
		\end{cases}
	\end{align}
\end{theorem}

\begin{remark}
	In view of the bound $\log \frac{\RMM}{\RQM} \leq C_{\rm QM}^{\rm MM}$ we
	could replace $\log \RMM$ with $\log\RQM$ in \eqref{error-u-f-mix}, however,
	we keep $\log \RMM$ to highlight the dependence of the error estimate on the
	growth of $\RMM$ relative to $\RQM$.
\end{remark}

\begin{proof}
	We will follow the same strategy as the proof of Theorem \ref{theorem-e-mix}.
	
	{\it 1. Quasi-best approximation: } 
	We take the approximation $T^{\rm H}\bar{u}\in\AH$ constructed
	in the proof of Theorem \ref{theorem-e-mix}, so that the
	properties from \eqref{proof-4-1-1} to \eqref{proof-4-1-3} are satisfied.
	
	{\it 2. Stability: }
	Let $\EH$ be defined by \eqref{E-H} with $\VMM$ being the 
	$(k+1)$-th order expansion in \eqref{site_MM_k}.
	For any $v\in\Usx$, we have
	\begin{eqnarray}\label{proof-f-s}
		\big\< \delta\FH(T^{\rm H}\bar{u}) v, v \big\> 
		= \big\< \big(\delta\FH(T^{\rm H}\bar{u}) 
		- \delta^2\EH(T^{\rm H}\bar{u})\big) v, v \big\> 
		+ \big\< \delta^2\EH(T^{\rm H}\bar{u}) v, v \big\> ,
	\end{eqnarray}
	where the first term is estimated in \ref{sec-proof-s-Q1} as
	\begin{eqnarray}\label{proof-f-s-Q1}
		\left|\big\< \big(\delta\FH(T^{\rm H}\bar{u}) 
		- \delta^2\EH(T^{\rm H}\bar{u})\big) v, v \big\>\right| \leq C 
		\left( \RQM^{-k+3/4} + e^{-\kappa\Rbuf} \right) \|D v\|^2_{\ell^2_{\gamma}}
	\end{eqnarray}
	with some constant $\kappa$,
	and the second term is estimated in \S \ref{sec-e-mix-idea} \eqref{e-mix-stability}.
	Therefore, we have that for sufficiently large $\RQM,~\RMM$ and $\Rbuf$,
	\begin{eqnarray}\label{f-mix-stability}
		\big\< \delta\FH(T^{\rm H}\bar{u}) v, v \big\>  
		\geq \frac{\bar{c}}{4} \| Dv \|_{\ell^2_\gamma}^2.
	\end{eqnarray}

	{\it 3. Consistency: }
	% Let $\EH$ be defined by \eqref{E-H} with $\VMM$ being the $(k+1)$-th
	% order expansion in \eqref{site_MM_k}.
	We estimate the consistency error for any $v\in\Usx$:
	\begin{eqnarray}\label{proof-4-1-4}
		\big\<\FH(T^{\rm H}\bar{u}) , v \big\>
		=\big\<\delta\EH(T^{\rm H}\bar{u}) , v \big\> + 
		\big\<\FH(T^{\rm H}\bar{u}) - \delta\EH(T^{\rm H}\bar{u}) , v \big\> ,
	\end{eqnarray}
	where the first term has been estimated in \S \ref{sec-e-mix-idea}
	and the second term can be written as
	\begin{eqnarray}\label{proof-4-1-5}
		\nonumber
		&& \big\<\FH(T^{\rm H}\bar{u}) - \delta\EH(T^{\rm H}\bar{u}) , v \big\>
		\\[1ex] \nonumber
		&=& \sum_{\ell\in\LQM\backslash\LI}\big( \FH_{\ell}(T^{\rm H}\bar{u})
		- \nabla_{\ell}\EH(T^{\rm H}\bar{u}) \big) v(\ell)
		+ \sum_{\ell\in\LI}\big( \FH_{\ell}(T^{\rm H}\bar{u})
		- \nabla_{\ell}\EH(T^{\rm H}\bar{u}) \big) v(\ell)
		\\[1ex] \nonumber
		&& + \sum_{\ell\in\LMM\backslash\LI}\big( \FH_{\ell}(T^{\rm H}\bar{u})
		- \nabla_{\ell}\EH(T^{\rm H}\bar{u}) \big) v(\ell)
		\\[1ex]
		&:=& P_1 + P_2 + P_3
	\end{eqnarray}
	with the interface region $\LI:=\{\ell\in\L,~\RQM-\Rbuf\leq|\ell|\leq\RQM+\Rbuf\}$.
	
	To estimate $P_1$, we have from the expressions \eqref{eq-force-Du-Omega}
	% with $\Omega=\LQM\backslash\Lbuf$ 
	that for any $\ell\in\LQM\backslash\LI$,
	\begin{eqnarray}\label{proof-5-c-1}
		\nonumber
		&& \F_{\ell}^{\rm H}(T^{\rm H}\bar{u}) - \nabla_{\ell}\EH(T^{\rm H}\bar{u}) 
		\\[1ex] \nonumber
		&=& \sum_{\ell-\rho\in\Lbuf} \Big(  
		V^{\LQM\cup\Lbuf}_{\ell-\rho,\rho}\big(Du_0(\ell-\rho)
		+ DT^{\rm H}\bar{u}(\ell-\rho)\big) 
		- \VMM_{\ell-\rho,\rho}\big(Du_0(\ell-\rho) 
		+ DT^{\rm H}\bar{u}(\ell-\rho)\big) \Big)
		\\ \nonumber
		&& - \sum_{\ell+\rho\in\Lbuf}  \Big( 
		V^{\LQM\cup\Lbuf}_{\ell,\rho}\big(Du_0(\ell) + DT^{\rm H}\bar{u}(\ell)\big)
		-  \VMM_{\ell,\rho}\big(Du_0(\ell) + DT^{\rm H}\bar{u}(\ell)\big) \Big) 
		\\[1ex]
		&\leq& Ce^{-\eta\Rbuf} ,
	\end{eqnarray}
	with some constant $\eta$,
	where Theorem \ref{theorem-thermodynamic-limit} (i)
	is used for the last inequality.
	Then we have from Lemma \ref{lemma-emb}
	%$\|v\|_{\ell^6}\leq C\|Dv\|_{\ell^2_{\gamma}}$ for $m=3$ and
	%$\|v\|_{\ell^{\infty}}\leq C\log\RMM\|Dv\|_{\ell^2_{\gamma}}$
	%for $m=2$ (which is derived from Lemma \ref{lemma-emb} and
	%$v\in\Usx$, see \eqref{proof-C-2-2}) 
	that when 
	$\Rbuf>\frac{4}{\eta}\log\RQM$ and $\Rbuf>\frac{4}{\eta}\log\log\RMM$,
	\begin{eqnarray}\label{proof-5-c-P1}
		%\sum_{\ell\in\LQM} \left\< \F_{\ell}^{\rm H}(T^{\rm H}\bar{u}) 
		%- \nabla_{\ell}\EH(T^{\rm H}\bar{u}) , v \right\> 
		P_1 \leq C e^{-\frac{\eta}{4}\Rbuf} \|Dv\|_{\ell^2_{\gamma}}.
	\end{eqnarray}
	
	To estimate $P_2$, we have
	\begin{eqnarray*}
		&& |\F_{\ell}^{\rm H}(T^{\rm H}\bar{u}) - \nabla_{\ell}\EH(T^{\rm H}\bar{u})|
		\\[1ex]
		&\leq & | \F^{\rm BUF}_{\ell}(T^{\rm H}\bar{u}) - \F_{\ell}(T^{\rm H}\bar{u}) | 
		+ |\nabla_{\ell}\E(T^{\rm H}\bar{u}) - \nabla_{\ell}\EH(T^{\rm H}\bar{u})|
		\\[1ex]
		&\leq & C \Big( e^{-\eta\Rbuf} +
		\sum_{\ell-\rho\in B_{\Rbuf}(\ell)\cap\LMM} e^{-\gamma|\rho|} \cdot |\ee(\ell-\rho) 
		+ \DD T^{\rm H}\bar{u}(\ell-\rho)|_{\gamma}^{k+1} \Big)
	\end{eqnarray*}
	for $\ell\in\LI\cap\LQM$ and
	\begin{eqnarray*}
		&& |\F_{\ell}^{\rm H}(T^{\rm H}\bar{u}) - \nabla_{\ell}\EH(T^{\rm H}\bar{u})|
		\\[1ex]
		&\leq & | \FH_{\ell}(T^{\rm H}\bar{u}) - \F^{\rm BUF}_{\ell}(T^{\rm H}\bar{u}) | 
		+| \F^{\rm BUF}_{\ell}(T^{\rm H}\bar{u}) - \F_{\ell}(T^{\rm H}\bar{u}) | 
		+ |\nabla_{\ell}\E(T^{\rm H}\bar{u}) - \nabla_{\ell}\EH(T^{\rm H}\bar{u})|
		\\[1ex]
		&\leq & C \Big( e^{-\eta\Rbuf} + 
		\sum_{\ell-\rho\in B_{\Rbuf}(\ell)\cap\LMM} e^{-\gamma|\rho|} \cdot |\ee(\ell-\rho) 
		+ \DD T^{\rm H}\bar{u}(\ell-\rho)|_{\gamma}^{k+1} \Big)
	\end{eqnarray*}
	for $\ell\in\LI\cap\LMM$.
	%
	% Then we have from
	% $\|v\|_{\ell^6}\leq C\|Dv\|_{\ell^2_{\gamma}}$ (for $m=3$) and
	% $\|v\|_{\ell^{\infty}}\leq C\log\RMM\|Dv\|_{\ell^2_{\gamma}}$
	% that when $\Rbuf>\frac{4}{\kappa}\log\RQM$,
	% \begin{eqnarray}\label{proof-5-c-P2}
	% P_2 \leq C \left( e^{-\kappa\Rbuf} + \RQM \right) \|Dv\|_{\ell^2_{\gamma}}.
	% \end{eqnarray}
	%
	Let $\L^{\rm I'}:=\{\ell\in\L,~\RQM-\Rbuf\leq|\ell|\leq\RQM+2\Rbuf\}$.
	If {\bf P} is satisfied, then we have from Theorem 
	\ref{th:appl:regularity} and Lemma \ref{lemma-emb} that
	\begin{eqnarray}\label{f-mix-consistency-P2-P}
		\nonumber
		P_2 &\leq& C \sum_{\ell\in\L^{\rm I'}} 
		\left( e^{-\eta\Rbuf} + |\ell|^{-m(k+1)} \right) 
		\Big( \sum_{|\rho|\leq\Rbuf} e^{-\gamma|\rho|} \cdot |v(\ell+\rho)| \Big)
		\\[1ex] 
		&\leq& C \|Dv\|_{\ell^2_{\gamma}} \left\{ \begin{array}{ll}
			\log\RMM \cdot \Rbuf \cdot
			\left( \RQM^{-2k-1} + \RQM \cdot e^{-\eta\Rbuf}  \right)
			& {\rm if~} m=2
			\\[1ex]
			\Rbuf^{5/6} \cdot \left( \RQM^{-3k-4/3} 
			+ \RQM^{5/3} \cdot e^{-\eta\Rbuf} \right)  & {\rm if~} m=3
		\end{array} \right. . \qquad
	\end{eqnarray}
	If {\bf D} is satisfied, then
	\begin{eqnarray}\label{f-mix-consistency-P2-D}
		\nonumber
		P_2 &\leq& C \sum_{\ell\in\L^{\rm I'}} 
		\left( e^{-\eta\Rbuf} + |\ell|^{-(k+1)} \right) 
		\Big( \sum_{|\rho|\leq\Rbuf} e^{-\gamma|\rho|} \cdot |v(\ell+\rho)| \Big)
		\\[1ex] 
		&\leq& C \log\RMM \cdot \Rbuf \cdot \left( \RQM^{-k}
		+ \RQM \cdot e^{-\eta\Rbuf} \right) \|Dv\|_{\ell^2_{\gamma}} .
		\quad
	\end{eqnarray}
	
	To estimate $P_3$, let $\Fa_{\ell}(v):=F_{\ell}(Px_0+v)$ and
	$\Ea(v):=\sum_{\ell\in\L}\big(E_{\ell}(Px_0+v)-E_{\ell}(Px_0)\big)$.  
	Define
	\begin{eqnarray}\label{Tk_tildeF}
		T_k \Fa_\ell(w) = \nabla_\ell T_{k+1} \Ea(w)
		:= \frac{\partial T_{k+1} \Ea(w)}{\partial w_\ell}.
		% := \sum_{\ell\in\L} \delta \Big( T_{k+1}V_{\ell}\big(Dw(\ell)\big) 
		% - T_{k+1}V_{\ell}\big(\ee(\ell)\big) \Big).
	\end{eqnarray}
	Then, for any $\ell\in\LMM\backslash\LI$,  we have
	\begin{eqnarray*}
		&& \hspace{-2cm} |\F_{\ell}^{\rm H}(T^{\rm H}\bar{u}) 
		- \nabla_{\ell}\EH(T^{\rm H}\bar{u})|
		\\[1ex]
		&\leq & | \FH_{\ell}(T^{\rm H}\bar{u}) - T_k\Fa_{\ell}(u_0+T^{\rm H}\bar{u}) | 
		+ |\nabla_{\ell}T_{k+1}\Ea(u_0+T^{\rm H}\bar{u})  
		- \nabla_{\ell}\EH(T^{\rm H}\bar{u})|
		\\[1ex]
		&\leq& \left\{ \begin{array}{ll}
			C e^{-\eta\Rbuf} |\ell|^{-m}  & {\rm if~\bf P} 
			\\[1ex]
			C e^{-\eta\Rbuf} |\ell|^{-2}\log|\ell| & {\rm if~\bf D} 
		\end{array} \right. ,
	\end{eqnarray*}
	where the same arguments as those in Lemma \ref{lemma-vlhr-vbuf}
	are used to derive the last inequality.
	Then we have from Lemma \ref{lemma-emb} that when
	$\Rbuf>\frac{4}{\eta}\log\RQM$ and	$\Rbuf>\frac{6}{\eta}\log\log\RMM$,
	\begin{eqnarray}\label{proof-5-c-P3}
		P_3 \leq C e^{-\frac{\eta}{4}\Rbuf} \|Dv\|_{\ell^2_{\gamma}} .
	\end{eqnarray}
	
	Taking \eqref{proof-4-1-4}, \eqref{proof-4-1-5}, \eqref{proof-5-c-P1},
	\eqref{f-mix-consistency-P2-P}, \eqref{f-mix-consistency-P2-D}, \eqref{proof-5-c-P3}
	and the estimates \eqref{e-mix-consistency-P}, \eqref{e-mix-consistency-D}
	with order $k+1$ into accounts, we have the consistency
	\begin{eqnarray}\label{f-mix-consistency}
		|\big\<\FH(T^{\rm H}\bar{u}) , v \big\>|
		\leq C\|Dv\|_{\ell^2_{\gamma}}
		\left\{ \begin{array}{ll}
			\RQM^{-(2k+1)m/2}\log\RMM + \RMM^{-m/2} + e^{-\frac{\eta}{4}\Rbuf} 
			&  {\rm if~\bf P}
			\\[1ex]
			\RQM^{-k}\log\RMM + \RMM^{-1}\log\RMM + e^{-\frac{\eta}{4}\Rbuf} 
			& {\rm if~\bf D}
		\end{array} \right. 
	\end{eqnarray}
	when $\Rbuf>\frac{4}{\eta}\log\RQM$ and $\Rbuf>\frac{4}{\eta}\log\log\RMM$.

	{\it 4. Application of inverse function theorem: }	
	With the stability \eqref{f-mix-stability} and consistency
	\eqref{f-mix-consistency}, we can apply the inverse function
	theorem \cite[Lemma 2.2]{ortner11} to obtain the existence of
	$\uH$ and the estimate
	\begin{eqnarray}\label{err-f-mix-u}
		\|D\uH-D\bar{u}\|_{\ell^2_{\gamma}} \leq 
		\left\{ \begin{array}{ll}
			C\Big( \RQM^{-(2k+1)m/2}\log\RMM + \RMM^{-m/2} 
			+ e^{-\kappa \Rbuf} \Big) &  {\rm if~\bf P}
			\\[1ex]
			C \Big( \RQM^{-k}\log\RMM + \RMM^{-1}\log\RMM 
			+ e^{-\kappa \Rbuf}  \Big)  & {\rm if~\bf D}
		\end{array} \right.  
	\end{eqnarray}
	with some constant $\kappa$. This completes the proof.
\end{proof}

\section{Concluding remarks}\label{sec-conclusions}
\setcounter{equation}{0}
In this paper, we construct new QM/MM coupling algorithms for crystalline solids
with embedded defects, based on either energy-mixing or force-mixing
formulations. Unlike in commonly used QM/MM schemes, our approach does not
employ ``off-the-shelf" interatomic potentials (or forces), but constructs a
potential (or force) specifically for the coupling with the QM model. The
accuracy of our algorithms (with respect to increasing QM region size) is
quantified by rigorous convergence rates. 

In the energy-based QM/MM coupling methods, with a given size $\RQM$ of the QM
region, we observe from Theorem \ref{theorem-e-mix} that one should take
$\RMM \approx \RQM^{\alpha/\beta}$ (e.g., in the case {\bf P}, $k=2$,
$\RMM \approx \RQM^3$) and $\Rbuf \approx \log\RQM$ to balance the errors. With
these choices, we obtain the errors in Table \ref{table-e-mix}, written in terms
of $\RQM$, dropping logarithmic contributions, and writing the order of
expansion as $k = k_{\rm E}$.

In our force-mixing QM/MM scheme, we obtain precisely the same rates and hence
the same balance of approximation parameters, except that the order of expansion
in the force is one less than that of the energy in our energy-mixing
scheme. The rates are also shown in Table \ref{table-e-mix}, with
$k = k_{\rm F}$.

\begin{table}
	\begin{center}
	\begin{tabular}{r|ccc|ccc|ccc}
		& \multicolumn{3}{l}{Case {\bf P}, $m=2$} &
		\multicolumn{3}{l}{Case {\bf P}, $m=3$} &
		\multicolumn{3}{l}{Case {\bf D}}  \\[1mm]
		\hline
		$k_{\rm E}$  &  2  & 3 & 4 & 2 & 3 & 4 & 2 & 3 & 4 \\
		$k_{\rm F}$  &  1  & 2 & 3 & 1 & 2 & 3 & 1 & 2 & 3 \\
		\hline 
		$\RMM$    &  $\RQM^3$ & $\RQM^5$ & $\RQM^7$ 
		& $\RQM^3$ & $\RQM^5$ & $\RQM^7$
		& $\RQM$ & $\RQM^2$ & $\RQM^3$  \\[1mm]
		error     &  $\RQM^{-3}$ & $\RQM^{-5}$ & $\RQM^{-7}$ 
		& $\RQM^{-4.5}$ & $\RQM^{-7.5}$ & $ \RQM^{-10.5}$ 
		& $\RQM^{-1}$ & $\RQM^{-2}$ & $\RQM^{-3}$ \\[1mm]
		E-error   &  $\RQM^{-4}$ & $\RQM^{-6}$ & $\RQM^{-8}$
		&  $\RQM^{-6}$ & $\RQM^{-9}$ & $\RQM^{-12}$
		& $\RQM^{-2}$ & $\RQM^{-3}$ & $\RQM^{-4}$ \\[1mm]
		\hline
	\end{tabular}
	\caption{Choice of $\RMM$ and error with respect to $\RQM$ for QM/MM
          schemes, with MM potential order $k = k_{\rm E}$ for the energy based
          scheme and $k = k_{\rm F}$ for the force-based scheme. The energy
          error applies only for energy-mixing schemes.}
	\label{table-e-mix}
	\end{center}
\end{table}

We note in particular that, for point defects, the QM/MM hybrid scheme achieves
dramatic rates of convergence, already for a second order expansion of the site
energies, respectively first order expansion of the forces
($k_{\rm E} = 2, k_{\rm F} = 1$). By contrast, for dislocations, the second
order expansion is no better than pure QM ``clamped boundary condition"
calculations (see \cite[\S 4.2]{chen15a} and \cite{ehrlacher13}). Only higher
order expansions ($k_{\rm E} \geq 3, k_{\rm F} \geq 2$) of the site energy will
give improved rates of convergence for hybrid QM/MM simulation of dislocations.

\begin{figure}
  \centering
  \includegraphics[width=\textwidth]{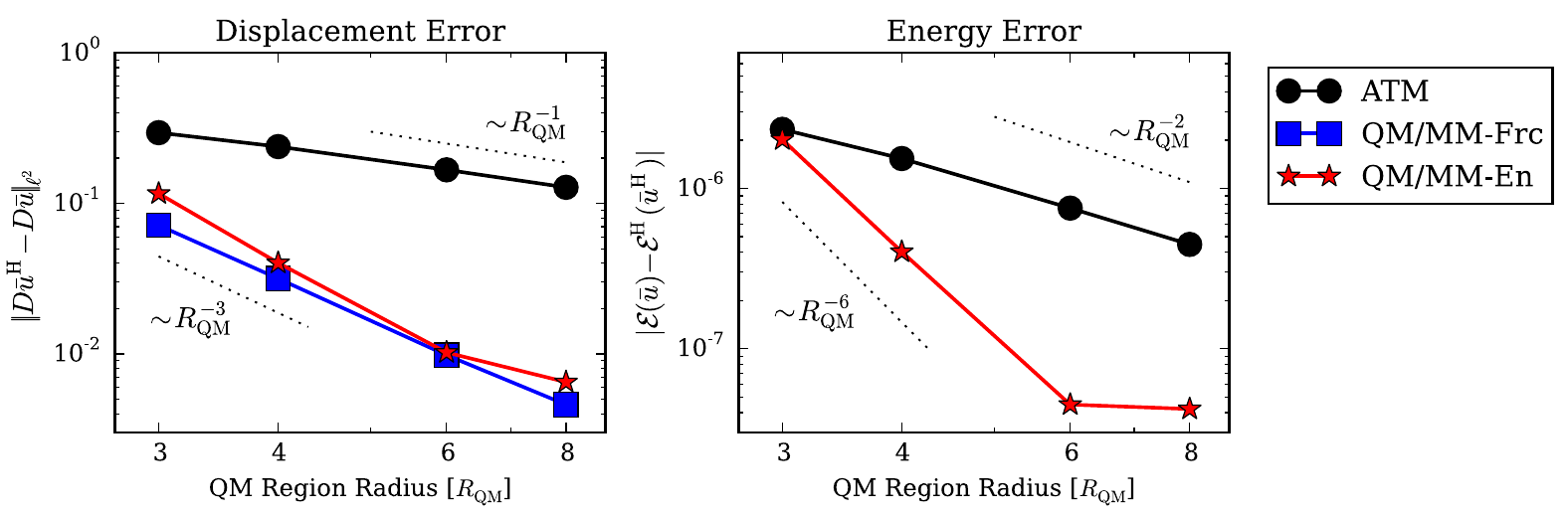}
  \caption{Numerical verification of the convergence rates predicted in Theorems
    \ref{theorem-e-mix} and \ref{theorem-f-mix} (ATM denotes a pure QM scheme as
    described in \cite{chen15a}). The results are consistent with the theory,
    but the numerical rate for the energy is better than our analytical
    prediction. (See \cite{ehrlacher13} for a similar gap in the theory).  The
    inconsistency in the rates in the last data point in the energy error, and
    to some extent also visible in the displacement error for QM-MM-En, is
    likely due to a buffer radius that is chosen slightly too small for this
    level of accuracy.}
  \label{fig:numerical-test}
\end{figure}

To limit the scope of the present work we will address the challenges in the
  implementation of both schemes in a separate article in full detail, but we
  present a preliminary numerical test. Using the TB toy model from
  \cite[Sec. 5]{chen15a}, the same simulation setup (2D triangular lattice with
  a di-vacancy defect), $k_{\rm E} = 2, k_{\rm F} = 1$, buffer radii
  $\Rbuf = 1 + 0.6 \log(\RQM)$ and MM domain radii
  $\RMM = \frac12 \RQM^3 + 2 \Rbuf$ we obtain numerical the results displayed in
  Figure \ref{fig:numerical-test}.  This test should only be considered as a
  motivation for further study, but its implementation allows us to make the
  following observations:

(1) A particular challenge in our schemes is the computational cost of higher
order expansions, which is of the order $O\big((\Rbuf)^{km}\big)$. For example,
taking only up to third neighbours in an FCC lattice
($\Rbuf / R_{\rm NN} \approx 1.7$, where $R_{\rm NN}$ is the nearest-neighbour
distance) results in 42 neighbouring atoms, which would result in over 2M
expansion coefficients at third order, and over 250M expansion coefficients at
fourth order. We will exploit lattice symmetries to reduce the number of
expansion coefficients that need to be calculated.  The fact that the order of
expansion is lower in force-based schemes, without loss of accuracy, is a
significant advantage.

% http://www.ifmpan.poznan.pl/~urbaniak/fcc%20and%20its%20neigbours01.pdf

(2) The computation of the $k$-th order expansion of the site energies requires
$k$-th order perturbation theory (or, finite-differences). By contrast, the
computation of forces and their derivatives can take advantage of the
``$2n+1$-Theorem'', hence expanding the forces is computationally much cheaper
than expanding energies, even at the same order of expansion. An analogous
comment applies to the computation of the QM region contribution to the hybrid
forces or gradient of the hybrid energy.

We conclude by commenting that, in view of the computational cost associated
with Taylor expansions as site energies, alternative approaches may be
required. Our analysis, or variations thereof, can then still be applied as long
as the MM model is tuned to interact ``correctly'' with the QM model.

% APPENDIX
\appendix
\renewcommand\thesection{\appendixname~\Alph{section}}

\section{Interpolation of lattice functions}
\label{sec-interpolation}
\renewcommand{\theequation}{A.\arabic{equation}}
\renewcommand{\thetheorem}{A.\arabic{theorem}}
\renewcommand{\thelemma}{A.\arabic{lemma}}
\renewcommand{\theproposition}{A.\arabic{proposition}}
\renewcommand{\thealgorithm}{A.\arabic{algorithm}}
\setcounter{equation}{0}

For each $u:\L\rightarrow\R^d$, we denote its continuous and
piecewise affine interpolant with respect to
$\mathcal{T}_{\L}$ by $Iu$,
% Identifying $u=Iu$, 
and its piecewise constant gradient by $\nabla Iu$.
We have the following lemma from \cite{ortner12,ortner13}.
\begin{lemma}\label{lemma-norm-interpolant}
	If $v\in\Wdot$, then there exist constants $c$ and $C$ such that
	\begin{eqnarray}\label{norm-equiv}
		c\|\nabla v\|_{L^2}\leq \|Dv\|_{\ell^2_{\gamma}} \leq C\|\nabla v\|_{L^2}.
	\end{eqnarray}
\end{lemma}

The following auxiliary results are useful in our analysis in that they
sometimes allow us to avoid stress-strain (``weak'') representations of residual
forces that we need to estimate.

\begin{lemma}\label{lemma-emb}
	(i) If $m = 2$, then there exists $C > 0$ such that
	\begin{align*}
		|v(\ell)-v(m)| \leq C \|Dv\|_{\ell^2_{\gamma}} \big(1+\log|\ell-m|\big) 
		\qquad \forall v \in \UsH, \quad  \ell, m \in \Lambda.
	\end{align*}
	(ii) If $m = 3$, then there exists $C > 0$ such that, for each $v \in \UsH$
	there exists $v_\infty \in \R^d$ such that
	\begin{displaymath}
		\| v-v_\infty \|_{\ell^6} \leq C \| Dv \|_{\ell^2}.
	\end{displaymath}
\end{lemma}
\begin{proof}
	The result is a straightforward generalisation of \cite[Proposition 12 (ii,
	iii)]{ortner12}.
\end{proof}

\begin{lemma}\label{lemma-f-Dv}
	Let $m=2$, $0<L<R$ and $v:\L\rightarrow\R$ satisfy $v(\ell)=0~\forall~|\ell|\geq R$.
	If $f:\L\rightarrow\R$ satisfies $|f(\ell)|\leq c|\ell|^{-2}$,
	then there exists a constant $C$ such that
	\begin{eqnarray}\label{estimate-f-Dv}
		\sum_{L\leq|\ell|\leq R} f(\ell)\cdot v(\ell)
		\leq C\log^{3/2}\left(\frac{R}{L}\right) \cdot \|Dv\|_{\ell^2_{\gamma}}
		\qquad\forall~v\in\Usx .
	\end{eqnarray}
\end{lemma}

\begin{proof}
	For simplicity of notation let $r := |{\bf r}|, \hat{\bf r} = {\bf r}/r$ and
	$v=Iv$.  Let $R' \geq R$, minimal, such that $v = 0$ in $B_{R'}^{\rm c}$.
	%  Note that
	% $|f(\ell)|\leq C|\ell|^{-2}$ implies $|If({\bf r})|\leq Cr^{-2}$ with
	% $r=|{\bf r}|$. 
	For each $T \in \mathcal{T}_\Lambda, T \subset B_{R'} \setminus B_L$ we have
	\begin{align*}
		\bigg| \sum_{\ell \in \Lambda \cap T} f(\ell)\cdot v(\ell) \bigg|
		\leq \max_{\ell \in \Lambda \cap T} |f(\ell)| \sum_{\ell \in \Lambda \cap T} 
		|v(\ell)| \leq C \max_{\ell \in \Lambda \cap T}|\ell|^{-2} 
		\int_T |Iv({\bf r})| \dd{\bf r} 
		\leq C \int_T r^{-2} |Iv({\bf r})| \dd{\bf r},
	\end{align*}
	Therefore, it follows that
	\begin{equation}
		\label{proof-A-2}
		\bigg|\sum_{L\leq|\ell|\leq R} f(\ell)\cdot v(\ell)\bigg|
		\leq C \int_{B_{R'}\backslash B_L} \frac{|Iv({\bf r})|}{r^2} \dd {\bf r}.
	\end{equation}
	
	%  \co{[Please replace $R$ with $R'$ in the following]} 
	We have from the estimate
	\begin{displaymath}
		|v({\bf r})| = \bigg| \int_{r}^{R'} \frac{d}{dt} v(t\hat{\bf r}) \dd t \bigg|
		\leq \int_{r}^{R'} |\nabla v(t\hat{\bf r})| \dd t.
	\end{displaymath}
	that
	\begin{align*}
		\int_{B_{R'}\backslash B_L} \frac{|v({\bf r})|}{r^2} \dd{\bf r}
		&= 
		\int_{S^{m-1}} \int_{r=L}^{R'} r^{-1}|v(r\hat{\bf r})| \dd r \dd{\hat{\bf r}} 
		\leq 
		\int_{S^{m-1}} \int_{r=L}^{R'} r^{-1} \int_{t = r}^{R'} |\nabla v(t\hat{\bf r})| \, 
		\dd t\, \dd r\, \dd\hat{\bf r}  \\
		&= 
		\int_{S^{m-1}} \int_{t = L}^{R'} |\nabla v(t\hat{\bf r})|  \int_{r = L}^tr^{-1}  \, 
		\dd r\, \dd t\, \dd\hat{\bf r}   \\
		&= 
		\int_{S^{m-1}} \int_{t = L}^{R'} t^{1/2}|\nabla v(t\hat{\bf r})| t^{-1/2}  
		\log \frac{t}{L} \, \dd t\, \dd\hat{\bf r}  \\
		&\leq \| \nabla v \|_{L^2(B_{R'}\backslash B_L)}  \bigg( \int_{S^{m-1}} 
		\int_{t = L}^{R'} t^{-1} \log^2 (t/L)\,\dd t \, \dd\hat{\bf r} \bigg)^{1/2} \\
		&\leq C \log^{3/2} \left(\frac{R'}{L}\right) \| \nabla v \|_{L^2}.
	\end{align*}
	Applying Lemma \ref{lemma-norm-interpolant}, \eqref{proof-A-2} and noting that
	$\log(R'/R) \leq C$ completes the proof.
\end{proof}

\section{Estimates of buffer truncations}
\label{sec-proof-c-T1}
\renewcommand{\theequation}{B.\arabic{equation}}
\renewcommand{\thetheorem}{B.\arabic{theorem}}
\renewcommand{\thelemma}{B.\arabic{lemma}}
\renewcommand{\theproposition}{B.\arabic{proposition}}
\renewcommand{\thealgorithm}{B.\arabic{algorithm}}
\setcounter{equation}{0}

\def\bg{{\bm g}}
\def\Vlh{V_{\ell}^{\frac{|\ell|}{3}}}
\def\Vlhr{V_{\ell,\rho}^{\frac{|\ell|}{3}}}
\def\VlA{V_{\ell}^{\rm A}}
\def\VlAr{V_{\ell,\rho}^{\rm A}}
\def\Vlbr{V^{\rm BUF}_{\#,\rho}}

The following lemma can be proved by a calculation with the
contour integration of the resolvent
(we refer the detailed proofs to the arguments in
\cite[(3.10)-(3.12)]{chen15a}).

\begin{lemma}\label{lemma-Vl-MN-err}
	Let $\ell\in M\subsetneq N\subset\L$.  Let
	$y\in\mathcal{V}_{\mathfrak{m}}(\L)$, $0\leq j\leq\mathfrak{n}-1$, and
	${\bm \rho}=(\rho_1,\cdots,\rho_j)\in (M-\ell)^j$.  Then there exist constants
	$C$ and $\eta$, depending on $\mathfrak{m}$, such that
	\begin{eqnarray}\label{Vl-MN-err}
		\left| 
		V^M_{\ell,\pmb{\rho}}\big(Dy(\ell)\big)
		- V^N_{\ell,\pmb{\rho}}\big(Dy(\ell)\big)
		\right| 
		\leq Ce^{ -\eta \left( {\rm dist}(\ell,N\backslash M) 
			+ \sum_{i=1}^{j}|\rho_i| \right)}
	\end{eqnarray}
	with ${\rm dist}(\ell,\Omega)=\min_{k\in \Omega}\{|\ell-k|\}$.
	Moreover, we have
	\begin{eqnarray}\label{Vl-Minf-err}
		\left| 
		V^M_{\ell,\pmb{\rho}}\big(Dy(\ell)\big)
		- V_{\ell,\pmb{\rho}}\big(Dy(\ell)\big)
		\right| 
		\leq Ce^{ -\eta \left( {\rm dist}(\ell,\L\backslash M) 
			+ \sum_{i=1}^{j}|\rho_i| \right) }.
	\end{eqnarray}
\end{lemma}

A direct consequence of Lemma \ref{lemma-Vl-MN-err} is 
\begin{eqnarray}\label{Vl-buf-err}
	\left| 
	V_{\ell,\pmb{\rho}}\big(Dy(\ell)\big)
	- V_{\ell,\pmb{\rho}}^{\rm BUF}\big(Dy(\ell)\big)
	\right| 
	\leq Ce^{ -\eta R_{\rm BUF} }
\end{eqnarray}
for $y\in\mathcal{V}_{\mathfrak{m}}(\L)$
and ${\bm \rho}=(\rho_1,\cdots,\rho_j)\in(\L-\ell)^j$
with $0\leq j\leq\mathfrak{n}-1$, $\max_i|\rho_i|\leq\Rbuf$.

For $y\in\mathcal{V}_{\mathfrak{m}}$, $\ell\in\LMM\cup\LFF$
and $R>0$, we define
$V_{\ell}^R\big(Dy(\ell)\big) := V_{\ell}^{B_R(\ell)\backslash B_{\Rcore}}\big(Dy(\ell)\big)$.
Therefore, we have
\begin{eqnarray}
	V^R_{\ell,\rho}\big(Dy(\ell)\big) := \left\{ \begin{array}{ll} 
		0, & {\rm if}~ |\rho|>R ~{\rm or}~ \ell+\rho\in \L\cap B_{\Rcore} \\[1ex]
		V_{\ell,\rho}^{B_R(\ell)\backslash B_{\Rcore}}\big(Dy(\ell)\big), & {\rm otherwise}
	\end{array}  \right. .
\end{eqnarray}
The difference between $V_{\ell}$ and $V_{\ell}^R$ can be estimated using Lemma
\ref{lemma-Vl-MN-err}:
If $R<|\ell|-\Rcore$, $0\leq j\leq\mathfrak{n}-1$, and
${\bm \rho} \in(\L-\ell)^j$ with $\max_i|\rho_i|\leq R$,
then there exist constants $C$ and $\eta$ such that 
\begin{eqnarray}\label{VlR-err}
	\left| 
	V_{\ell,\pmb{\rho}}\big(Dy(\ell)\big)
	- V_{\ell,\pmb{\rho}}^{R}\big(Dy(\ell)\big)
	\right| 
	\leq Ce^{ -\eta \left( R
		+ \sum_{i=1}^{j}|\rho_i| \right) }.
\end{eqnarray}

The next lemma establishes the homogeneity of the site energy $V_\ell^R$.

\begin{lemma}\label{lemma-hom}
	Let $\ell,k\in\L$ and $R<\min\{|\ell|,|k|\}-\Rcore$.  If
	$y,y'\in \mathcal{V}_{0}(\L)$ satisfy $D_{\rho}y(\ell)=D_{\rho}y'(k)$ for
	all $|\rho|<R$, then
	\begin{eqnarray}\label{hom-VlR}
		V_{\ell}^{R}\big(Dy(\ell)\big) = V_{k}^{R}\big(Dy'(k)\big)
		\qquad{\rm and}\qquad
		V_{\ell,{\bm\rho}}^{R}\big(Dy(\ell)\big) =
		V_{k,{\bm\rho}}^{R}\big(Dy'(k)\big)
	\end{eqnarray}
	for ${\pmb \rho}\in(A\Z^m\cap B_R-0)^j$, $\max_i|\rho_i|<R$.
\end{lemma}  

\begin{proof}
	Using the condition 
	$\{D_{\rho}y(\ell)\}_{|\rho|<R} =
	\{D_{\rho}y'(k)\}_{|\rho|<R}$
	and
	Theorem \ref{theorem-thermodynamic-limit} 
	(ii) with $g(x)=x-x(\ell)+x(k)$,
	(iii) with $\mathcal{G}(x)=x-\ell+k$, 
	we can derive $V_{\ell}^{R}\big(Dy(\ell)\big) = V_{k}^{R}\big(Dy'(k)\big)$.
	Then the second part of \eqref{hom-VlR} is a direct consequence.
\end{proof}

Before the proof of \eqref{proof-c-T1}, we need the following
estimate for $\VlAr-\Vlbr$ on the predictor, where the auxiliary
site potential $\VlA$ is defined by
\begin{eqnarray*}
	\VlA({\bf g}) := \left\{ \begin{array}{ll}
		V_{\ell}^{|\ell|-\Rcore}({\bf g}) & {\rm if}~|\ell|\leq 3\RQM \\[1ex]
		V_{\ell}^{\frac{|\ell|}{3}}({\bf g}) & {\rm if}~|\ell|> 3\RQM
	\end{array}\right.
	\quad {\rm for}~{\bf g}\in\big(\R^d\big)^{\L-\ell} .
\end{eqnarray*}

\begin{lemma}\label{lemma-vlhr-vbuf}
	Let $\Rbuf>\max\{\frac{4}{\eta}\log\RQM,\frac{6}{\eta}\log\log\RMM\}$,
	where $\eta$ is the constant from Lemma \ref{lemma-Vl-MN-err}.  
	If the assumption {\bf P} or {\bf D} is satisfied, 
	then there exists a constant $C$ such that
	\begin{eqnarray}\label{vlhr-vbuf}
		\sum_{\ell\in\LMM\cup\LFF} 
		\big\< \delta V_{\#}^{\rm BUF}\big(\ee(\ell)\big) 
		- \delta \VlA\big(\ee(\ell)\big) , \DD v(\ell) \big\> 
		\leq  Ce^{-\frac{\eta}{4}}  \|Dv\|_{\ell^2_{\gamma}} 
		\qquad \forall v \in \Usx.
	\end{eqnarray}	
\end{lemma}

\begin{proof}
	The left-hand side of \eqref{vlhr-vbuf} can be written in the form
	\begin{multline}\label{proof-A-2-0}
		\sum_{\ell\in\LMM\cup\LFF} \sum_{\rho\in\L-\ell}
		% \sum_{\rho\in A\Z^d-0,|\rho|<\frac{|\ell|}{3}}
		\left( V_{\#,\rho}^{\rm BUF}\big(Du_0(\ell)\big) 
		- \VlAr\big(Du_0(\ell)\big) \right)
		\big( v(\ell+\rho) - v(\ell) \big) 
		\\[1ex]
		:= \sum_{\RQM-\Rcore \leq |\ell| \leq 3\RQM}
		\big( \tilde{f}^{\rm A}(\ell) 
		- \tilde{f}^{\rm BUF}(\ell) \big) v(\ell) 
		+ \sum_{|\ell|>3\RQM} \big(
		f^{\rm A}(\ell) - f^{\rm BUF}(\ell) \big) v(\ell),
		\quad
	\end{multline}
	where $f^{\rm A}(\ell)$ and
	$\tilde{f}^{\rm A}(\ell)$ are
	given in terms of $\VlAr(Du_0(\ell))$,
	$f^{\rm BUF}(\ell)$ and $\tilde{f}^{\rm BUF}(\ell)$ are
	given in terms of $V^{\rm BUF}_{\ell,\rho}(Du_0(\ell))$.
	The precise forms of $\tilde{f}^{\rm A}(\ell)$ and
	$\tilde{f}^{\rm BUF}(\ell)$ are not important, 
	we can obtain from Lemma \ref{lemma-Vl-MN-err} and the fact
	that they are given in terms of $\VlAr(Du_0(\ell))$ and
	$V^{\rm BUF}_{\ell,\rho}(Du_0(\ell))$ that
	\begin{eqnarray}\label{proof-A-2-2}
		\left| \tilde{f}^{\rm A}(\ell) - \tilde{f}^{\rm BUF}(\ell) \right|
		\leq C e^{-\eta\Rbuf} 
		\qquad{\rm for}~\RQM-\Rcore \leq |\ell|\leq 3\RQM.
	\end{eqnarray}
	
	When $|\ell|>3\RQM$, we have
	\begin{eqnarray}
		\label{proof-A-2-1}
		f^{\rm BUF}(\ell) &=& \sum_{\rho\in A\Z^d-0,~|\rho|\leq\Rbuf}
		\left( V_{\ell-\rho,\rho}^{\rm BUF}\big(Du_0(\ell-\rho)\big)
		- V_{\ell,\rho}^{\rm BUF}\big(Du_0(\ell)\big) \right) 
		\qquad {\rm and}
		\\[1ex]
		\nonumber
		f^{\rm A}(\ell)
		&=& \sum_{\rho\in A\Z^d-0,~|\rho|\leq\frac{|\ell|}{3}}
		\left( V_{\ell-\rho,\rho}^{\rm A}\big(Du_0(\ell-\rho)\big)
		- V_{\ell,\rho}^{\rm A}\big(Du_0(\ell)\big) \right)
		\\[1ex]
		\nonumber
		&=& \sum_{\rho\in A\Z^d-0,~|\rho|\leq\frac{|\ell|}{3}}
		\left( V_{\ell-\rho,\rho}^{\rm A}\big(Du_0(\ell-\rho)\big)
		- V_{\ell-\rho,\rho}^{\frac{|\ell|}{3}}\big(Du_0(\ell-\rho)\big) \right)
		\\[1ex]
		\nonumber
		&& + \sum_{\rho\in A\Z^d-0,~|\rho|\leq\frac{|\ell|}{3}}
		\left( V_{\ell-\rho,\rho}^{\frac{|\ell|}{3}}\big(Du_0(\ell-\rho)\big)
		- V_{\ell,\rho}^{\frac{|\ell|}{3}}\big(Du_0(\ell)\big) \right)
		% \\[1ex] \nonumber	&& \hc{ [[here is why we need to use |\ell|/3]]}
		\\[1ex]
		\label{proof-A-2-8}
		&:=& f_a(\ell) + f_b(\ell).
	\end{eqnarray}
	Lemma \ref{lemma-Vl-MN-err} and the definition of $\VlA$ 
	imply that
	\begin{eqnarray}\label{proof-A-2-6}
		f_a \leq C e^{-\frac{2}{9}\eta|\ell|}.
	\end{eqnarray}
	
	If {\bf P} is satisfied, then we have from $u_0(\ell-\rho)
	=u_0(\ell)=0$ and Lemma \ref{lemma-hom} that
	\begin{eqnarray}\label{proof-A-2-7}
		f_b(\ell)=0	\qquad{\rm and}\qquad f^{\rm BUF}(\ell)=0.
	\end{eqnarray}
	
	If {\bf D} is satisfied, we first consider the left halfspace
	$\ell_1<\hat{x}_1$, in which case we can replace $\DD$ by $D$.
	Let $e=Du_0(\ell)$ and suppressing the argument $(\ell)$,
	we have from Lemma \ref{lemma-hom} and the expansion
	\begin{eqnarray}\label{proof-A-2-4}
		V_{,\rho}(e) = V_{,\rho}({\bf 0}) + 
		\big\< \delta V_{,\rho}({\bf 0}) , e \big\> +
		\frac{1}{2}\int_0^1(1-t)\big\< \delta^2 V_{,\rho}(te)e , e \big\>\dd t 
		\quad{\rm with}\quad V=\Vlh~{\rm or}~V_{\#}^{\rm BUF} ~~
	\end{eqnarray}
	that
	\begin{multline}\label{proof-A-2-11}
		f_b(\ell) - f^{\rm BUF}(\ell)
		= \sum_{\rho,\xi} \left( V^{\frac{|\ell|}{3}}_{\ell,\rho\xi}({\bf 0})
		- V^{\rm BUF}_{\#,\rho\xi}({\bf 0}) \right) D_{-\rho}e_{\xi}
		\\
		+ \frac{1}{2}\int_0^1(1-t) D_{-\rho} \left\< 
		\Big( \delta^2 V^{\frac{|\ell|}{3}}_{\ell,\rho}(te)
		- \big(\delta^2 V^{\rm BUF}_{\#,\rho}(te) \Big) e , e \right\>\dd t 
		:= F_1+F_2 .   \quad
	\end{multline}
	Using Lemma \ref{lemma-decay-el} and Lemma \ref{lemma-Vl-MN-err},
	we have
	\begin{eqnarray*}
		F_1 \leq Ce^{-\eta\Rbuf} |\ell|^{-2} 
		\qquad{\rm and}\qquad
		F_2 \leq Ce^{-\eta\Rbuf} |\ell|^{-2} ,
	\end{eqnarray*}
	which implies $|f_b(\ell) 
	- f^{\rm BUF}(\ell)|\leq Ce^{-\eta\Rbuf} |\ell|^{-2}$
	for $\ell_1<\hat{x}_1$.
	For the right halfspace $\ell_1>\hat{x}_1$, we can repeat the foregoing
	argument to deduce
	\begin{eqnarray*}
		\left|S\left( f_b(\ell) - f^{\rm BUF}(\ell) \right)\right| 
		\leq Ce^{-\eta\Rbuf} |\ell|^{-2} .
	\end{eqnarray*}
	Note that $S$ is an $O(1)$ shift, which implies
	$|f_b(\ell) - f^{\rm BUF}(\ell)|
	\leq Ce^{-\eta\Rbuf} |\ell|^{-2}$ for $\ell_1>\hat{x}_1$.
	
	Taking \eqref{proof-A-2-1}, \eqref{proof-A-2-8}, \eqref{proof-A-2-6},
	\eqref{proof-A-2-7} and the above estimates for
	$|f_b(\ell) - f^{\rm BUF}(\ell)|$
	into accounts, we have
	\begin{eqnarray}\label{proof-A-2-3}
		\left| f^{\rm A}(\ell) - f^{\rm BUF}(\ell) \right|
		\leq C \left( e^{-\frac{2}{9}\eta|\ell|} + e^{-\eta\Rbuf} |\ell|^{-2}  \right) 
		\qquad{\rm for}~~|\ell|>3\RQM.
	\end{eqnarray}

	Combining \eqref{proof-A-2-0}, \eqref{proof-A-2-2},
	\eqref{proof-A-2-3} and $v\in\Usx$ yields
	\begin{eqnarray}\label{proof-A-2-12}
		&&\sum_{\ell\in\LMM\cup\LFF} 
		\big\< \delta V_{\#}^{\rm BUF}\big(\ee(\ell)\big) 
		- \delta \VlA\big(\ee(\ell)\big) , \DD v(\ell) \big\> 
		= \sum_{\RQM-\Rcore\leq |\ell|\leq\RMM} \mathfrak{f}(\ell) v(\ell)
		\\
		% \end{eqnarray}	
		% \begin{eqnarray}
		\label{esitmate-fl}
		\text{with} &&	|\mathfrak{f}(\ell)| \leq \left\{ \begin{array}{ll}
			Ce^{-\eta\Rbuf} & {\rm if} ~ \RQM-\Rcore \leq |\ell| \leq 3\RQM \\[1ex]
			C \left( e^{-\frac{2}{9}\eta|\ell|} + e^{-\eta\Rbuf} |\ell|^{-2}  \right) 
			& {\rm if} ~ |\ell|>3\RQM
		\end{array} \right. .
	\end{eqnarray}
	Since $\Rbuf>\frac{4}{\eta}\log\RQM$, we can write 
	$|\mathfrak{f}(\ell)| \leq Ce^{-\frac{\eta}{2}\Rbuf} |\ell|^{-2}$,
	which together with Lemma \ref{lemma-f-Dv} completes the
	proof of \eqref{vlhr-vbuf} as $\Rbuf>\frac{6}{\eta}\log\log\RMM$.
\end{proof}

\begin{proof}
	[Proof of \eqref{proof-c-T1}]
	Denote $g(\ell)=\DD T^{\rm H}\bar{u}(\ell)$ and suppressing the argument
	$(\ell)$ in $g(\ell)$ and $\ee(\ell)$, we have from Lemma
	\ref{lemma-Vl-MN-err} and Lemma \ref{lemma-vlhr-vbuf} that
	\begin{align}\label{proof-A-1-9}
		\nonumber
		\sum_{\ell\in\LMM\cup\LFF} 
		\big\< \delta V_{\#}^{\rm BUF}\big(\ee+g\big) 
		- \delta \VlA\big(\ee+g\big) , \DD v(\ell) \big\>  
		% \\[1ex]
		% \nonumber
		&= \sum_{\ell\in\LMM\cup\LFF} 
		\big\< \delta V_{\#}^{\rm BUF}\big(\ee\big) 
		- \delta \VlA\big(\ee\big) , \DD v(\ell) \big\>
		\\[1ex]
		\nonumber
		& \hspace{-6cm} + \sum_{\ell\in\LMM\cup\LFF} \int_0^1(1-t)
		\left\< \Big( \delta^2 V_{\#}^{\rm BUF}\big(\ee+tg\big) 
		- \delta^2 \VlA\big(\ee+tg\big) \Big) g, \DD v(\ell) \right\>   \dd t 
		\\[1ex]
		&\hspace{-1cm} \leq C \left( 
		e^{-\frac{\eta}{4}\Rbuf} + e^{-\eta\Rbuf} \|g\|_{\ell^2_{\gamma}}   \right)
		%+ e^{-\eta\Rbuf}  \|g\|^2_{\ell^4_{\gamma}}
		\|Dv\|_{\ell^2_{\gamma}} .
	\end{align}
	
	Using \eqref{Vl-buf-err}, \eqref{VlR-err}, \eqref{proof-A-1-9}
	and Theorem \ref{th:appl:regularity}, we have	
	\begin{eqnarray}\label{proof-A-2-20}
		\nonumber
		T_1 & = & \sum_{\ell\in\LQM} 
		\big\< \delta V_{\ell}^{\rm BUF}\big(\ee+g\big) 
		- \delta V_{\ell}\big(\ee+g\big) , \DD v(\ell) \big\> 
		\\[1ex]
		\nonumber
		&& 
		+ \sum_{\ell\in\LMM\cup\LFF} 
		\big\< \delta V_{\#}^{\rm BUF}\big(\ee+g\big) 
		- \delta \VlA\big(\ee+g\big) , \DD v(\ell) \big\> 
		\\[1ex]
		\nonumber
		&& + \sum_{\ell\in\LMM\cup\LFF} 
		\big\< \delta \VlA\big(\ee+g\big) 
		- \delta V_{\ell}\big(\ee+g\big) , \DD v(\ell) \big\> 
		%\\[1ex] \nonumber
		%&\leq & C \sum_{\ell\in\LQM} e^{-\eta\Rbuf} |\DD v(\ell)|_{\gamma}
		%+ + C \sum_{\ell\in\LMM\cup\LFF} e^{-\frac{1}{3}\eta|\ell|} |\DD v(\ell)|_{\gamma}
		\\[1ex]
		&\leq & C\big( \RQM^{m/2}e^{-\eta\Rbuf}  + e^{-\frac{\eta}{4}\Rbuf}
		+ e^{-\frac{\eta}{4}\RQM} \big) \|Dv\|_{\ell^2_{\gamma}} ,
	\end{eqnarray}	
	which completes the proof since $\RQM^{m/2}e^{-\eta\Rbuf}$
	and $e^{-\frac{\eta}{4}\RQM}$ can be omitted compared with
	the middle term when $\Rbuf>\frac{m}{2\eta}\log\RQM$.
\end{proof}

\begin{proof}
	[Proof of \eqref{proof-e-3}.]
	Using $T^{\rm H}\bar{u}\in\AH$ and the expansion
	\begin{eqnarray}\label{proof-A-3-1}
		V(\ee+g) = V(\ee) + \big\< \delta V(\ee) , g \big\> +
		\frac{1}{2}\int_0^1(1-t)\big\< \delta^2 V(\ee+tg)g , g \big\>\dd t 
		~{\rm with}~ V=V_{\ell}~{\rm or}~V_{\ell}^{\rm BUF} ~~~
	\end{eqnarray}
	for $\RQM\leq|\ell|\leq\RMM+\Rbuf$, we have
	\begin{eqnarray}\label{proof-A-3-2}
		\nonumber
		|S_1| &\leq &
		\sum_{\ell\in \LQM} 
		\left( V_{\ell}(g+\ee) - V_{\ell}(\ee)
		- V^{\rm BUF}_{\ell}(g+\ee) + V^{\rm BUF}_{\ell}(\ee) \right) 
		\\ \nonumber
		&& + \sum_{\RQM\leq|\ell|\leq\RMM+\Rbuf} 
		\big\< \delta V_{\ell}(\ee) - \delta V_{\ell}^{\rm BUF}(\ee) , g \big\> 
		\\ \nonumber
		&& + \frac{1}{2}\sum_{\RQM\leq|\ell|\leq\RMM+\Rbuf} 
		\int_0^1(1-t)\left\< \big(\delta^2 V_{\ell}(\ee+tg) 
		- \delta^2 V_{\ell}^{\rm BUF}(\ee+tg)\big) g , g \right\>\dd t 
		\\[1ex]
		&:=& S_1^a + S_1^b + S_1^c .
	\end{eqnarray}
	The first and third group of \eqref{proof-A-3-2} can be
	estimated by Lemma \ref{lemma-Vl-MN-err}:
	\begin{eqnarray}\label{proof-A-3-3}
		|S_1^a|\leq C\RQM^{m/2}e^{-\eta\Rbuf}
		\qquad{\rm and}\qquad
		|S_1^c|\leq Ce^{-\eta\Rbuf}\|g\|^2_{\ell^2_{\gamma}}.
	\end{eqnarray}
	Using Lemma \ref{lemma-vlhr-vbuf}, 
	%    \co{[to me this seems essentially an application of that lemma?]}, 
	we can obtain the estimate for the second group:
	\begin{eqnarray*}
		|S_1^b|\leq Ce^{-\frac{\eta}{4}\Rbuf}
		\qquad{\rm when}~\Rbuf>\frac{6}{\eta}\log\log\RMM,
	\end{eqnarray*}
	which together with \eqref{proof-A-3-2} and \eqref{proof-A-3-3}
	completes the proof of \eqref{proof-e-3}.
\end{proof}

\section{Stability of force-mixing methods}
\label{sec-proof-s-Q1}
\renewcommand{\theequation}{C.\arabic{equation}}
\renewcommand{\thetheorem}{C.\arabic{theorem}}
\renewcommand{\thelemma}{C.\arabic{lemma}}
\renewcommand{\theproposition}{C.\arabic{proposition}}
\renewcommand{\thealgorithm}{C.\arabic{algorithm}}
\setcounter{equation}{0}

\def\Tu{T^{\rm H}\bar{u}}
\def\tF{\widetilde{\F}}
\def\tFb{\widetilde{\F}^{\rm b}}
\def\tEb{\widetilde{\E}^{\rm b}}

Here, we establish the result that the energy-mixing Hessian and force-mixing
Jacobian are ``close''. This result is reminiscent of similar results in the
context of atomistic/continuum blending \cite{luskin13}, but the proofs are not
closely related.

\begin{proof}[Proof of \eqref{proof-f-s-Q1}]
	Let $\LI:=\{\ell\in\L,~\RQM-\Rbuf\leq|\ell|\leq\RQM+\Rbuf\}$
	be the interface region.
	Denoting
	% $w:=T^{\rm H}\bar{u}$,
	$\tFb_{\ell}(v):=\widetilde{\F}^{B_{\Rbuf}(\ell)}_{\ell}(v)$ and 
	$\tEb(v):= \sum_{\ell\in\L}\big(E^{\rm  buf}_{\ell}(Px_0+v)-E^{\rm BUF}_{\ell}(Px_0)\big)$
	with $E_{\ell}^{\rm BUF}=E_{\ell}^{B_{\Rbuf}(\ell)}$, 
	we can split 
	\begin{eqnarray}\label{proof-C-1}
		\nonumber
		&& \hspace{-1.5cm} \big\< \big(\delta\FH(T^{\rm H}\bar{u}) 
		- \delta^2\EH(T^{\rm H}\bar{u})\big) v, v \big\>
		% \\[1ex] \nonumber
		=
		\sum_{\ell\in \L}\big\<\delta\FH_{\ell}(\Tu)-\delta\nabla_{\ell}\EH(\Tu), v\big\> v(\ell) 
		\\[1ex] \nonumber
		&=& 
		\sum_{\ell\in A\Z^m}\big\<\delta\tFb_{\ell}(0)-\delta\nabla_{\ell}\tEb(0), v\big\> v(\ell) 
		\\[1ex] \nonumber
		&& + \sum_{\ell\in \LQM\backslash\LI}\big\<\delta\FH_{\ell}(\Tu)
		-\delta\nabla_{\ell}\EH(\Tu), v\big\> v(\ell) 
		\\[1ex] \nonumber
		&& - \sum_{\ell\in A\Z^d\cap B_{\RQM-\Rbuf}}\big\<\delta\tFb_{\ell}(0)
		-\delta\nabla_{\ell}\tEb(0), v\big\> v(\ell) 
		\\[1ex] \nonumber
		&& + \sum_{\ell\in \LI}\left( \big\<\delta\FH_{\ell}(\Tu)
		-\delta\tFb_{\ell}(0), v\big\> v(\ell) 
		- \big\<\delta\nabla_{\ell}\EH(\Tu)-\delta\nabla_{\ell}\tEb(0), v\big\> v(\ell)  \right)
		\\[1ex] \nonumber
		&& + \sum_{\ell\in\LMM\backslash\LI}\left( \big\<\delta\FH_{\ell}(\Tu)
		-\delta\tFb_{\ell}(0), v\big\> v(\ell) 
		- \big\<\delta\nabla_{\ell}\EH(\Tu)-\delta\nabla_{\ell}\tEb(0), v\big\> v(\ell)  \right)
		\\[1ex]
		&:=& Q_{1} + Q_{2} + Q_{3} + Q_{4} + Q_{5} .
	\end{eqnarray}

	{\it Estimate for $Q_1$: } Using Theorem \ref{theorem-thermodynamic-limit-force}
	(ii), (iii) and \cite[Lemma 3.4]{hudson12}, we can rewrite $Q_{1}$ as
	\begin{eqnarray}\label{proof-Q1-1}
		Q_{1} = \sum_{\ell\in A\Z^d}\sum_{\rho\in A\Z^d-0} 
		D_{\rho}v(\ell)^{\rm T} \big(A^{F}_{\rho}-A^{E}_{\rho}\big) D_{\rho}v(\ell)
	\end{eqnarray}
	with $A^{F}_{\rho},A^{E}_{\rho}\in\R^{d\times d}$,
	$A^{F}_{\rho}=-\frac{1}{2}\tFb_{0,\rho}(0)$ and
	$A^{E}_{\rho}=-\frac{1}{2}\tEb_{,0\rho}(0)$. This crucially uses that fact
	that we apply the force approximation $\widetilde{\mathcal{F}}_\ell^{\rm b}$
	at {\em every lattice site}, which makes it conservative.  Note that Lemma
	\ref{lemma-Vl-MN-err} and Theorem \ref{theorem-thermodynamic-limit} (i) imply
	% and the expressions \eqref{eq:force-Du}, \eqref{eq-force-Du-Omega}
	\begin{eqnarray*}
		|A^{F}_{\rho}-A^{E}_{\rho}| \leq Ce^{-\eta(\Rbuf+|\rho|)},
	\end{eqnarray*}
	which together with \eqref{proof-Q1-1} yields
	\begin{eqnarray}\label{proof-C-Q1}
		|Q_{1}| \leq Ce^{-\eta\Rbuf} \|Dv\|^2_{\ell^2_{\gamma}}.
	\end{eqnarray}
	
	{\it Estimates for $Q_2, Q_3$: } Lemma \ref{lemma-emb} and $v\in\Usx$ imply that
	\begin{equation}\label{proof-C-2-2}
		\begin{split}
			\|v\|_{\ell^{\infty}} &\leq C\|Dv\|_{\ell^2_{\gamma}} 
			\big( 1+\log\RMM \big), \qquad \text{if } m = 2, \text{ and}  \\
			\|v\|_{\ell^6} &\leq C \|Dv\|_{\ell^2_{\gamma}}, \qquad \text{if } m = 3.
		\end{split}
	\end{equation}
	Using \eqref{proof-C-2-2}, Lemma \ref{lemma-Vl-MN-err} and 
	% \eqref{eq:force}, \eqref{eq-force-Du-Omega}, 
	Theorem \ref{theorem-thermodynamic-limit} (i), we have
	\begin{eqnarray}\label{proof-C-Q2}
		\nonumber
		|Q_2| &\leq& C\sum_{\ell\in \LQM\backslash\LI} 
		e^{-\eta\Rbuf} |Dv(\ell)|_{\gamma} |v(\ell)|
		\\[1ex] \nonumber
		&\leq& C \left\{\begin{array}{ll}
			\RQM\cdot \log\RMM \cdot e^{-\eta\Rbuf} \|Dv\|^2_{\ell^2_{\gamma}} 
			& {\rm if}~m=2
			\\[1ex]
			\RQM \cdot e^{-\eta\Rbuf} \|Dv\|^2_{\ell^2_{\gamma}} 
			& {\rm if}~m=3 
		\end{array} \right.
		\\[1ex]
		&\leq&  C\RQM \cdot \log\RMM \cdot e^{-\eta\Rbuf} \|Dv\|^2_{\ell^2_{\gamma}}.
	\end{eqnarray}
	Analogously, we have
	\begin{eqnarray}\label{proof-C-Q3}
		|Q_3| \leq C\RQM \cdot \log\RMM \cdot e^{-\eta\Rbuf} \|Dv\|^2_{\ell^2_{\gamma}} .
	\end{eqnarray}

	{\it Estimate for $Q_4$: } 
	We can rewrite $Q_{4}$ as
	\begin{align*} 
		& Q_{4} = \sum_{\ell\in \LI}\big\<\delta\FH_{\ell}(\Tu) 
		- \delta\F^{\rm BUF}_{\ell}(\Tu), v\big\> v(\ell) 
		+ \sum_{\ell\in \LI}\big\<\delta\F^{\rm BUF}_{\ell}(\Tu) 
		- \delta\tFb_{\ell}(u_0+\Tu), v\big\> v(\ell) 
		\\[1ex]
		& - \sum_{\ell\in \LI} \big\<\delta\nabla_{\ell}\EH(\Tu) 
		- \delta\nabla_{\ell}\E^{\rm BUF}(\Tu), v\big\> v(\ell) 
		- \sum_{\ell\in \LI}\big\<\delta\nabla_{\ell}\E^{\rm BUF}(\Tu) 
		- \delta\nabla_{\ell}\tEb(u_0+\Tu), v\big\> v(\ell) 
		\\[1ex] 
		& + \sum_{\ell\in \LI}\big\<\delta\tFb_{\ell}(u_0+\Tu)
		-\delta\tFb_{\ell}(0), v\big\> v(\ell) 
		- \sum_{\ell\in \LI}\big\<\delta\nabla_{\ell}\tEb(u_0+\Tu)
		-\delta\nabla_{\ell}\tEb(0), v\big\> v(\ell).
	\end{align*}
	Using \eqref{eq-force-Du-Omega} and
	$\big|V_{\ell-\rho,\rho\xi\zeta}^{B_{\Rbuf}(\ell)} -
	V_{\ell-\rho,\rho\xi\zeta}^{B_{\Rbuf}(\ell-\rho)}\big| \leq Ce^{-\eta
		\big(\big|\Rbuf-|\rho|\big|+|\rho|+|\xi|+|\zeta|\big)}$
	(this is proved analogously to Lemma \ref{lemma-Vl-MN-err}), the last line an be
	bounded by
	\begin{align*}
		& \Bigg| \sum_{\ell\in \LI} \sum_{\rho\in\ell-\LI,~|\rho|\leq\Rbuf} \Big\< 
		\left( \delta V^{B_{\Rbuf}(\ell)}_{\ell-\rho,\rho}\big(\ee(\ell-\rho)+\DD\Tu(\ell-\rho)\big) 
		- \delta V^{B_{\Rbuf}(\ell)}_{\ell-\rho,\rho}({\bf 0}) \right) 
		\\[1ex]
		& - \left( \delta V^{B_{\Rbuf}(\ell-\rho)}_{\ell-\rho,\rho}\big(\ee(\ell-\rho)+\DD\Tu(\ell-\rho)\big) 
		- \delta V^{B_{\Rbuf}(\ell-\rho)}_{\ell-\rho,\rho}({\bf 0}) \right) , 
		\DD v(\ell-\rho) \Big\> ~v(\ell) \Bigg|
		\\[1ex]
		&= \Bigg| \sum_{\ell\in\LI} \sum_{|\rho|\leq\Rbuf} \int_0^1 \Big\< \Big( \delta^2 V_{\ell-\rho,\rho}^{B_{\Rbuf}(\ell)}\big(t(\ee(\ell-\rho)+\DD\Tu(\ell-\rho))\big)
		\\[1ex] 
		& - \delta^2 V_{\ell-\rho,\rho}^{B_{\Rbuf}(\ell-\rho)}\big(t(\ee(\ell-\rho)+\DD\Tu(\ell-\rho))\big) \Big) 
		\big(\ee(\ell-\rho)+\DD\Tu(\ell-\rho)\big)  , \DD v(\ell-\rho) \Big\> \dd t ~ v(\ell) \Bigg|  
		\\[1ex]
		&\leq C \sum_{\ell\in\L^{\rm I'}} e^{-\eta \Rbuf} \cdot |\ee(\ell)+\DD\Tu(\ell)|_{\gamma} 
		\cdot |\DD v(\ell)|_\gamma \cdot \Big( \sum_{|\rho|\leq\Rbuf}e^{-\gamma|\rho|}\cdot|v(\ell+\rho)| \Big) .
	\end{align*}
	with $\L^{\rm I'}:=\{\ell\in\L, ~\max\{0,\RQM-2\Rbuf\} 
	\leq |\ell| \leq \RQM+2\Rbuf\}$.
	Using Theorem \ref{th:appl:regularity} and \eqref{proof-C-2-2}, we have 
	\begin{eqnarray}\label{proof-C-Q4}
		\nonumber
		|Q_{4}| &\leq& C \sum_{\ell\in\L^{\rm I'}} 
		\left( |\ee(\ell)+\DD\Tu(\ell)|_{\gamma}^k + e^{-\eta\Rbuf} \right)  
		\cdot |\DD v(\ell)|_{\gamma} \cdot \Big( \sum_{|\rho|\leq\Rbuf}e^{-\gamma|\rho|}\cdot|v(\ell+\rho)| \Big) 
		\\[1ex] \nonumber
		&\leq& C \Rbuf^{1/2} \cdot \|Dv\|^2_{\ell^2_{\gamma}} \left\{\begin{array}{ll}
			\log\RMM \big( \RQM^{-2k+1/2} + e^{-\eta\Rbuf} \big) & {\rm if ~ \bf P ~ \rm with~} m=2 \\[1ex]
			\RQM^{-3k+1/3} + e^{-\eta\Rbuf} & {\rm if ~ \bf P ~ \rm with~} m=3 \\[1ex]
			\log\RMM \big(\RQM^{-k+1/2} + e^{-\eta\Rbuf} \big) & {\rm if ~ \bf D} 
		\end{array} \right.
		\\[1ex] 
		&\leq& C\Rbuf^{1/2} \cdot \log\RMM \cdot \big(\RQM^{-k+1/2} + e^{-\eta\Rbuf} \big)  \|Dv\|^2_{\ell^2_{\gamma}}.
	\end{eqnarray}

	{\it Estimate for $Q_{5}$: } Let $\Fa_{\ell}(v):=F_{\ell}(Px+v)$ and
	$\Ea(v):=\sum_{\ell\in\L}\big(E_{\ell}(Px_0+v)-E_{\ell}(Px_0)\big)$, then
	$\Fa_\ell(v) = \nabla_\ell \Ea(v)$.  Further, we define
	\begin{align}\label{Tk_tildeF}
		\widehat{T}_k \Fa_\ell(w) &= \nabla_\ell \widehat{T}_{k+1} \Ea(w)
		:= \frac{\partial \widehat{T}_{k+1} \Ea(w)}{\partial w_\ell}
		\qquad \text{where} 
		\\[1ex] \nonumber
		\widehat{T}_{k+1}\Ea(w) &= T_{k+1}\Ea(w)- \Ea(0) - \big\< \delta\Ea(0) , w \big\> .
	\end{align}
	Then, for any $\ell\in\LMM\backslash\LI$, we have
	\begin{eqnarray*}
		&& \hspace{-2cm} \left|\Big\<\big(\delta\FH_{\ell}(\Tu) - \delta\tFb_{\ell}(0) \big)
		- \big(\delta\nabla_{\ell}\EH(\Tu) - \delta\nabla_{\ell}\tEb(0) \big) , v \Big\> \right|
		\\[1ex]
		&\leq & \left| \Big\< \delta\FH_{\ell}(\Tu) - \delta\tFb_{\ell}(0) 
		- \delta\widehat{T}_k \Fa_\ell(u_0+\Tu) , v \Big\> \right| \\
		&& \quad 
		+ \left| \Big\< \delta\nabla_{\ell}\EH(\Tu) - \delta\nabla_{\ell}\tEb(0) 
		- \delta\nabla_\ell \widehat{T}_{k+1} \Ea(u_0+\Tu) , v \Big\> \right|
		\\[1ex]
		&\leq& Ce^{-\eta\Rbuf} \sum_{|\rho|\leq\Rbuf}
		e^{-\gamma|\rho|}\cdot|\DD v(\ell-\rho)|_{\gamma} \cdot 
		\left\{ \begin{array}{ll}
			|\ell-\rho|^{-m}  & {\rm if~\bf P} 
			\\[1ex]
			|\ell-\rho|^{-2}\log|\ell-\rho| & {\rm if~\bf D} 
		\end{array} \right. .
	\end{eqnarray*}
	Summing over $\ell \in \LMM \setminus \LI$, and applying \eqref{proof-C-2-2} we
	obtain
	\begin{eqnarray}\label{proof-C-Q5}
		|Q_{5}| \leq C \,\log^2\frac{\RMM}{\RQM} \cdot e^{-\eta\Rbuf} \|Dv\|^2_{\ell^2_{\gamma}} .
	\end{eqnarray}
	(For case {\bf P} one obtains $\log \frac{\RMM}{\RQM}$ instead of
	$\log^2\frac{\RMM}{\RQM}$, but this is qualitatively the same as replacing the
	unknown exponent $\eta$ with $\eta/2$, hence we ignore this small
	improvement.)

	Combining \eqref{proof-C-1}, \eqref{proof-C-Q1}, \eqref{proof-C-Q2},
	\eqref{proof-C-Q3}, \eqref{proof-C-Q4} and \eqref{proof-C-Q5}, we finally deduce
	that
	\begin{eqnarray*}
		\left|\big\< \big(\delta\FH(T^{\rm H}\bar{u}) 
		- \delta^2\EH(T^{\rm H}\bar{u})\big) v, v \big\>\right| \leq C 
		\left( \RQM^{-k+3/4} + e^{-\frac{\eta}{4}\Rbuf} \right) \|D v\|^2_{\ell^2_{\gamma}}
	\end{eqnarray*}
	provided that $\RQM>\log^4\RMM$, $\Rbuf>\frac{3}{\eta}\log\RQM$,
	$\Rbuf>\frac{3}{\eta}\log\log\RMM$.  This completes the proof.
	% with $\kappa=\frac{\eta}{3}$.
\end{proof}

\bibliographystyle{siam}
\bibliography{qmmmtb2bib}

\end{document}